\documentclass[12pt, a4paper]{article}

\usepackage{amsthm}
\usepackage{cite}
\usepackage{mathtools}
\usepackage{enumitem}
\usepackage{hyperref}
\usepackage{graphicx}
\usepackage{amssymb}
\usepackage{xcolor}
\usepackage{multirow}
\usepackage[all]{xy}
\usepackage{tikz}
\usepackage{booktabs}
\usepackage{arydshln}
\usepackage{mleftright}
\usepackage{pgffor}
\usepackage{blkarray}
\usepackage{subfigure}
\usepackage{authblk}
\usepackage[utf8]{inputenc}
\usepackage[T1]{fontenc}
\usepackage{amsmath}
\usepackage{amsfonts}
\usepackage{caption}
\usepackage{float}
\usepackage{ifthen}
\usepackage{tikz-cd}
\usepackage{longtable}
\usepackage{fullpage}
\usepackage{textcomp}
\usepackage{makecell}
\usepackage{setspace}
\usepackage{color}
\usepackage{chngcntr}
\usepackage[misc]{ifsym}
\usepackage{stmaryrd}
\usepackage{mathabx}

\hypersetup{
  colorlinks,
  linkcolor={blue!50!black},
  citecolor={blue!50!black},
  urlcolor={blue!80!black}
}

\usetikzlibrary{patterns}
\usetikzlibrary{intersections}
\usetikzlibrary{calc}

\setlist[enumerate]{labelsep=*, leftmargin=1.5pc}
\setlist[enumerate]{label=\normalfont(\roman*), ref=\roman*}
\usepackage[margin=2.7cm]{geometry}
\graphicspath{{images/}}
\theoremstyle{plain}
\newtheorem{thm}{Theorem}[section]

\newtheorem{lem}[thm]{Lemma}
\newtheorem{rmk}[thm]{Remark}
\newtheorem{cor}[thm]{Corollary}
\newtheorem{conjecture}[thm]{Conjecture}
\theoremstyle{definition}
\newtheorem{dfn}[thm]{Definition}

\newtheorem{eg}[thm]{Example}

\newtheorem{con}[thm]{Construction}

\DeclareMathOperator{\Spec}{Spec}
\DeclareMathOperator{\Proj}{Proj}
\DeclareMathOperator{\Ext}{Ext}

\newcommand {\ddlog}[1]{\frac{d}{d \; log \: #1}}
\newcommand {\ddx}[1]{\frac{d}{d #1}}
\newcommand{\Hom}{\cH om}

\def \Alphabet {A,B,C,D,E,F,G,H,I,J,K,L,M,N,O,P,Q,R,S,T,U,V,W,X,Y,Z}
\def \alphabet {a,b,c,d,e,f,g,h,i,j,k,l,m,n,o,p,q,r,s,t,u,v,w,x,y,z}

\foreach \x in \Alphabet {\expandafter\xdef\csname\x\x\endcsname{\noexpand\mathbb{\x}}}
\foreach \x in \Alphabet {\expandafter\xdef\csname b\x\endcsname{\noexpand\mathbf{\x}}}
\foreach \x in \Alphabet {\expandafter\xdef\csname c\x\endcsname{\noexpand\mathcal{\x}}}
\foreach \x in \alphabet {\expandafter\xdef\csname fr\x\endcsname{\noexpand\mathfrak{\x}}}
\foreach \x in \Alphabet {\expandafter\xdef\csname fr\x\endcsname{\noexpand\mathfrak{\x}}}

\begin{document}

\title{Towards the Doran-Harder-Thompson conjecture via the Gross-Siebert program}
\author{Lawrence J. Barrott, Charles F. Doran\footnote{The second author was supported by the Natural Sciences and Engineering Research Council of Canada (NSERC)}}
\maketitle

\begin{abstract}

  The Doran-Harder-Thompson ``gluing/splitting'' conjecture unifies mirror symmetry conjectures for Calabi-Yau and Fano varieties, relating fibration structures on Calabi-Yau varieties to the existence of certain types of degenerations on their mirrors. This was studied for the case of Calabi-Yau complete intersections in toric varieties in~\cite{tciDHT} for the Hori-Vafa mirror construction. In this paper we prove one direction of the conjecture using a modified version of the Gross-Siebert program. This involves a careful study of the implications within tropical geometry and applying modern deformation theory for singular Calabi-Yau varieties.
  
\end{abstract}

Mirror symmetry started with the study of Calabi-Yau manifolds, predicting a duality associating to any Calabi-Yau $X$ a mirror partner $\check{X}$. As it grew it came to incorporate other statements, most notably it predicts a similar duality associating to any Fano variety and anti-canonical divisor $(Z,D)$ a Landau-Ginzburg model, a variety $\check{W}$ and a super-potential $w: \check{W} \to \CC$. The theory of Landau-Ginzburg models was formalised in~\cite{HarderLGModels} where these were required to be a fibration of $\check{X}$ by lower dimensional Calabi-Yau varieties.

The two worlds, Calabi-Yau and Fano, meet in the example of Tyurin degenerations. A Tyurin degeneration is a semi-stable degeneration $\frX_t \to \Delta$ such that the central fibre is the union of two smooth components $\frX_0 = \frZ_1 \cup \frZ_2$ meeting along a smooth divisor $\frD$. From the work of Kawamata-Nakayama $\frD$ is an effective anti-canonical divisor on the components $\frZ_i$, and the normal bundles are dual, $\cN_{\frD/\frZ_0} \otimes \cN_{\frD/\frZ_0} \cong \cO_{\frD}$, and this is a necessary and sufficient condition for such smoothings to exist. This was expanded on by Kato in~\cite{LogDeformationTheory} where he proved that this condition is equivalent to the existence of a log structure of semi-stable type (in modern language locally free) on the central fibre, log smooth over the standard log point. This generalises the Kawamata-Nakayama condition since there are surgeries permitted by log geometry which produce central fibres necessarily smoothable but not of the type they consider, where the central fibre is a chain of components. The most basic example would be to blow up $D$ inside the total space, to obtain a smoothable central fibre with three smooth components, $\cZ_0 \cup \cZ_1 \cup \cZ_2$ meeting along two copies of $\cD$, say $\cD_1$ and $\cD_2$. Note that it is not true that we have a balancing condition $\cN_{\cD_1 / \cZ_0} \otimes \cN_{\cD_1 / \cZ_1} \cong \cO_{\cD_1}$, instead we only have $\cN_{\cD_1 / \cZ_0} \otimes \cN_{\cD_1 / \cZ_1} \otimes \cN_{\cD_2 / \cZ_1} \otimes \cN_{\cD_2 / \cZ_2} \cong \cO_{\cD}$.

We will study higher order degenerations extending the definition of Tyurin degenerations. This produces a notion of higher rank Landau-Ginzburg model, appearing in situations where we have not a smooth divisor pair, but an snc divisor pair.

\begin{dfn}

  Let $\frX \to \AA^1$ be a family such that $\frX_{gen}$ and $\frX$ are both smooth and the central fibre is an snc divisor whose smallest strata is codimension $k$. We call such families $k$-Tyurin degenerations, in analogy to type $k$ degenerations of Calabi-Yau varieties.

   Let $\cS$ be a smooth variety of dimension $k$ and $w: \frW \to \cS$ a proper family with $\frW$ smooth and generic fibre a smooth Calabi-Yau. We call such a structure a rank $k$ Landau-Ginzburg model.
  
\end{dfn}

There is relatively little literature on higher rank LG models. A sensible restriction is to look at those families such that $\cS$ admits an snc compactification and such that $\frW$ admits a relatively log smooth compactification whose fibres over an $n$ dimensional strata of $\cS$ has a strata of dimension $n$. This recreates the expectation that the compactification of an LG model has a MUM point at infinity. We will see how this condition naturally arises. For the remains of the introduction we exposit the Tyurin case where $k=1$.

Under these conditions we can form the mirror to $\frX_{gen}$, $\widecheck{\frX_{gen}}$, and the mirror to the pairs $(\frZ_0, \frD)$ and $(\frZ_1, \frD)$, $W_0: \widecheck{\frZ}_0 \to \CC$ and $W_1: \widecheck{\frZ}_1 \to \CC$. One direction of the Doran-Harder-Thompson conjecture is the statement that $\widecheck{\frX_{gen}}$ is fibred over $\PP^1$, which we prove given a compatible toric degeneration:

\begin{conjecture}[``Degenerations to Fibrations'']
\label{ConjectureDHT} [Corollary~\ref{DegenerationsToFibrations}]
  If $\frX$ admits a Tyurin degeneration to $\frZ_0 \cup_{\frD} \frZ_1$ then the mirror $\widecheck{\frX_{gen}}$ admits a map to $\PP^1$ with fibres smooth Calabi-Yau varieties of one dimension lower and further $\frD$ is mirror to a generic fibre of this fibration.

\end{conjecture}

\noindent Our proof naturally extends to $k$-Tyurin degenerations, and demonstrates that they are fibred over a $k$-dimensional base.

A reconstruction theorem of the following flavour was found to hold in~\cite{DHTPaper},~\cite{Kanazawa}. These cases have the property that the fibres have at most one dimensional complex moduli, and this seems to be a necessary condition to have a holomorphic gluing, forming a second part of the conjecture:

\begin{conjecture}[``Reconstruction via gluing'']
\label{ConjectureDHT} [Theorem~\ref{ReconstructionByGluing}]
  Suppose that the fibres of $\widecheck{\frZ}_0 \to \AA^1$ have one dimensional complex moduli. Then the Calabi-Yau $\widecheck{\frX_{gen}}$ is obtained by gluing the LG models $\widecheck{\frZ}_0$ to $\widecheck{\frZ}_1$ along the fibres over $\CC^* \subset \CC$ under a deformation of the map $t \mapsto t^{-1}$.

\end{conjecture}

This conjecture has geometric interpretations on the level of periods for Calabi-Yau complete intersections in toric varieties using techniques from~\cite{DualPolyhedraAndMirrorSymmetryForCalabiYauHypersurfacesInToricVarieties}. As stated the conjecture is agnostic of how we construct the mirror, one typically uses the Batyrev-Borisov construction of~\cite{BatyrevBorisov} or the Hori-Vafa construction of~\cite{HoriVafa}. We will use an older version of the Gross-Siebert program, which assumes the existence of a toric degeneration:

\begin{thm}

  Suppose that $\frX \to \AA^1$ is a simple toric degeneration. Then the mirror to $\frX_{gen}$ is given by an explicit log smoothing of a combinatorially determined central fibre. It is a smoothing of a singular space over a monoid ring of effective curve classes in $\frX$.
  
\end{thm}

We define a compatibility of toric degenerations and Tyurin degenerations, and in general between toric and type $k$ degenerations with smooth total space. From this we find a fibration of a component of the mirror before we smooth. It is not clear at the moment how to simultaneously smooth both $\widecheck{\frX_0}$ and a Calabi-Yau subvariety inside it purely using the techniques of the program. The algebra map on the central fibre is clear, but there are quantum corrections that occur during the smoothing procedure which we cannot control. 

Thankfully there has been a large amount of progress in the past year on how to smooth log Calabi-Yau varieties using scattering diagrams. For us the most notable are recent advances first by Chan-Leung-Ma in~\cite{CLM} and then by Felten-Filip-Ruddat in~\cite{FFR}. These perturb the Maurer-Cartan equation in an iterative way to construct a formal, and then analytic, smoothing. The problem of smoothing complexes was studied in~\cite{CM} where the authors give a smoothing criterion for locally free complexes, which is not satisfied by our examples. Following classical deformation theory we prove the smoothability of a Calabi-Yau variety and a Calabi-Yau subvariety with trivial normal bundle by directly relating the BV-algebras used in the correction process. This also shows that if the normal bundle is not trivial then the subspace is at least stable. This result is perhaps of independent interest for constructing subvarieties tropically:

\begin{thm}[Theorem ~\ref{PairSmoothing}]

  Let $\frX$ be a log toroidal family over $\Spec k^\dagger$, and $\frD \subset \frX$ a strictly embedded complete intersection Calabi-Yau subvariety transverse to the singularities of the log structure on $\frX$. Then there are smoothings of the pair $\frX, \frD$ over $\Spec k[\![t]\!]$
  
\end{thm}

This statement could be proved by other means, showing that the normal bundle in this case describes locally isotrivial smoothings. Our eventual goal would be to construct a scattering process that extends this result to the intrinsic construction of~\cite{Intrinsic}. This requires that we study an appropriate differential graded Lie algebra (dgLa) via this more complicated machinery. 

Finally we apply relative deformation theory to the interior of the LG models and the fibration. This process only applies on the level of formal geometry, as relative deformation theory over an affine base is not very well behaved in the algebraic or analytic categories. We prove that $\widecheck{\frX_0}$ restricted to the interior $\GG_m \subset \AA^1$ is a versal deformation space for families of $\widecheck{\frD_0}$-bundles under a surjectivity assumption that can be checked on the mirror. If $\widecheck{\frD_0}$ has one dimensional deformation space this allows us to produce an open formal embedding of $\frW_0$ into $\frX$.

In the case of $k$-Tyurin degenerations the analogous statement requires a surjectivity $H^1 (\Theta_{\widecheck{\frX}}) \to H^1 (\Theta_{\widecheck{\frD}})$. For the simple toric degenerations we consider this can be studied on the mirror.

\subsubsection*{Comments on restrictions}

There are four comments we must make about this work. Firstly we do not know how to construct a quasi-Fano variety directly from a Landau-Ginzburg model. There is the work of Prince in~\cite{Scaffolds} which has partial solutions but this does not apply in the generality we are looking for. This makes the converse of the theorem impossible to resolve using these techniques in complete generality.

On the other hand the statement that the mirror to a fibred Calabi-Yau admits a Tyurin degeneration is tractable. In this case there is a canonical affine direction on the base of the mirror, corresponding to pairing with a fibre class. Once one smooths perpendicularly to this one should obtain a log scheme which is log smooth over the standard log point, which by~\cite{LogDeformationTheory} is enough to produce a Tyurin degeneration. This requires a much stronger analysis of the local structure of partial smoothings of the log structure than appears in say~\cite{FromRealAffine}.

The third comment is that we are restricted here to working in the case where $\frX$ is not dimension two. This is to avoid a polarisation issue that is intrinsic to the geometry of $K3$ surfaces. So long as the tropicalisation $\Sigma (\frX)$ has vanishing $H^2$ every choice of open gluing data for the fan picture gives rise to a choice of gluing data for the cone picture and our construction is written directly in terms of the cone picture. A more detailed study of the $K3$ surface case and its subtleties will be performed in~\cite{K3GS}.

Since there are already subtleties of mirror symmetry for $K3$ surfaces that have not been addressed in the wider literature we pass on addressing this.

The final comment is that it is possible to prove that the mirror is fibred directly from the recent paper~\cite{Intrinsic}. Unfortunately it is not immediately clear what it is fibred by. Essentially the toric degeneration assumption here forces enough local rigidity to allow us to study the scattering diagram through other deformation theoretic means. Historically this has precedent, the work of~\cite{Intrinsic} came only after the work of~\cite{FromRealAffine}. In future work we intend to use the gluing calculations of~\cite{Punctured} and~\cite{Wu} and the period calculations of~\cite{RuddatSiebert} to prove the period gluing formulae of~\cite{tciDHT} from the perspective of tropical curves, even though we do not know a precise method to prove that fibres are mirror to $\frD$.

\subsubsection*{Acknowledgement} We thank Mark Gross, Jordan Kostiuk, Helge Ruddat, Bernd Siebert, Alan Thompson, Yixian Wu and Fenglong You for helpful discussions during the development of this project. L. J. B. was supported by Boston College. C. F. D. was supported by the University of Alberta, the Center for Mathematical Sciences and Applications at Harvard University and NSERC.

\section{From the general fibre inwards}

We begin with an adapted version of toric Calabi-Yau degenerations to provide compatibility between toric and Tyurin degenerations.

\begin{dfn}

  Let $\cS = \Spec k[\![t_0, t_1]\!]$ be the formal spectrum of a two dimensional local ring with the log structure induced by the divisors $t_0$ and $t_1$. A toric connected $k$-Tyurin degeneration over $\cS$ is a proper normal log algebraic space over flat over $\cS$ with the following properties:

  \begin{enumerate}
  \item The generic fibre $\frX_\eta$ is an irreducible normal variety.
  \item The fibre over $t_1 = 0$ is reducible and the components are in bijection with the components over a closed point $t_0 = a \neq 0$.
  \item The restriction of $\frX$ to the open subscheme $\Spec k((t_0))[\![t_1]\!]$ is a $k$-Tyurin degeneration. In particular the closed fibre of this family $\Spec k((t_0))[\![t_1]\!]/\langle t_1 \rangle$ contains a minimal smooth strata $\frD_{gen}$ of codimension $k$, the intersection of components $\frZ^0_{gen}$ up to $\frZ^k_{gen}$. These have well defined boundary divisors, which we write $\partial \frZ^i_{gen}$  
  \item The central fibre $\frX_0$ satisfies the assumptions (2) and (3) of~\cite{LogDegenerationI} Definition 4.1.
  \item The map $\frX \to \cS$ is a saturated and log smooth morphism away from a set $\frS \subset \frX$ which is relative codimension $\leq 2$ and does not contain any toric strata of the central fibre $\frX_{0}$.
  \end{enumerate}

\end{dfn}

The fact that the general fibre is a $k$-Tyurin degeneration forces a very particular structure on the log structure, at least near points of $\frD$.

\begin{lem}

  There are charts for the log structure on $\frX$ in \'etale neighbourhoods of the points of the closure of $\frD_{gen}$ of the form:

  \[ t_1 \mapsto p \quad t_2 \mapsto y_0 \ldots y_k\]

\noindent  for $p, y_0, \ldots, y_k$ generating a monoid subring and such that $y_0, \ldots , y_k$ are irreducible elements cutting out $\frZ_{gen}^0$ up to $\frZ_{gen}^k$ generating an $\NN^k$ submonoid and not dividing $p$.
  
\end{lem}

\begin{proof}

  Since $\frX \to \cS$ is generically log smooth we can find charts: \newpage
  \begin{figure}[h]
    \centering
    \begin{tikzcd}
      \frX \arrow[bend left = 20]{rrd} \arrow[bend right = 30]{rdd} \arrow{rd}{\phi}& & \\
      & \Spec k[P] \times_{\Spec k[\NN^2]} \cS \arrow{r} \arrow{d} & \Spec k[P] \arrow[d] \\
      & \cS \arrow [r] & \Spec k[\NN^2]
    \end{tikzcd}
  \end{figure}

  \noindent where the map $\NN^2 \to P$ is an inclusion, $\phi$ is smooth and $t_1$ and $t_2$ correspond to the monomials $(1,0)$ and $(0,1)$ respectively and $\cS \to \Spec k[\NN^2]$ is smooth. Suppose that we can write $f^\# (t_1)$ as a sum of irreducible elements $\sum a_i p_i$. These irreducible elements give components of the fibre over $t_2 \neq 0$, with $a_i$ being the generic length of the subscheme $t_i = 0$ over each component. But we assumed that on this locus the degeneration had smooth total space, which enforces the above description of the monoid map, where $y_i$ is the monomial $z^{p_i}$.

  Now if $p$ were divisible by any $p_i$ then there would be an entire divisor contained in the fibre over $t_1 = t_2 = 0$. But by assumption $f$ is flat. The ring elements $y_i$ are the pullbacks of $p_i$.
  
\end{proof}

\begin{dfn}

We set $\frD$ to be the subscheme of $\frX$ obtained locally by the equations $y_i = 0$ for all $i$. We write $\frZ^i$ for the subscheme obtained locally by the equation $y_i = 0$.

\end{dfn}

We need to further restrict, so that the central fibre has degenerate properties similar to a Tyurin degeneration. For us this means that $\frD$ is a normal integral scheme, and this is satisfied so long as $\frX$ has property $S_{k+1}$ near $\frD$, in particular if the total space of the degeneration is Cohen-Macaulay since we know $\frD$ is $R_0$ and $R_0$ plus $S_1$ implies reduced. This is also satisfied for our examples arising as complete intersections in toric varieties since toric varieties are always Cohen-Macaulay.

\begin{dfn}
  
  We say that $\frX$ is a toric $k$-Tyurin degeneration if the induced toric degeneration of $\frD$ has normal integral total space and the Tyurin singular locus $\frD$ is transverse to the singular locus of the total space. We say that $\frZ^i$ is a \emph{toric degenerating component} if the induced degeneration of $\frZ^i$ is toric.
  
\end{dfn}

Many of the schemes appearing here are Cohen-Macaulay where being normal is equivalent to being regular in codimension one, which is substantially easier to check, we will use this observation without mention.

Taking the fibre of $\frX$ over the closed subset $t_0 = a \: t_1$ for generic choice of $a$ (which we may take to be 1) we obtain a toric degeneration in the sense of~\cite{LogDegenerationI} Definition 4.1 over a DVR. We write this degeneration $\frX_\Delta$. We can also extract from this the data of a toric degeneration of the components of the Tyurin degeneration.

\begin{lem}

  The scheme $\frD$ is a toric degeneration under the induced maps to $V (t_0 - a \: t_1)$.
  
\end{lem}

\begin{proof}

  Condition (1) of~\cite{LogDegenerationI} Definition 4.1 is trivial, the generic fibre $\frD_{gen}$ is smooth.

  To prove condition (2) we have to recall the construction and definition of the normalisation from~\cite{StacksProject}. We proceed in two steps. The first restricts to the vanishing of one of the $y_i$, which interacts with the normalisation in an essentially trivial way. The second step is then a stability under passing along a regular sequence.

  We fix $y_1$ say, and consider the associated degeneration of $\frZ_1$. The central fibre $\frX_0$ has normalisation the disjoint union of toric varieties $X_{\Xi_i}$ for $i \in I$ and the central fibre of $\frZ_i$ selects a subset $J \subset I$ of the components. The connected and irreducible components $X_{\Xi_i}$ are in bijection with the irreducible components of $\frX_0$ via a birational map with function field $K_i$. $X_{\Xi_i}$ is locally the integral closure of the corresponding component of $\frX_0$ inside the function field $K_i$.

  By the universal property of\cite[\href{https://stacks.math.columbia.edu/tag/035Q}{Lemma 035Q}]{StacksProject} the normalisation of $\frZ$ is the disjoint union $\coprod_{j \in J} X_{\Xi_j}$. The conductor locus is precisely the union of the toric Weil divisors not lying on an intersection between the components of $\frZ_1$. In particular it is reduced, the map $C_{\frZ_1} \to \nu (C_{\frZ_1})$ unramified and generically two-to-one. That the square:
      \begin{figure}[h]
    \centering
    \begin{tikzcd}
      C_{\frZ_1} \arrow{r} \arrow{d} & \prod _{j \in J} X_{\Xi_j} \arrow{d} \\
      \nu (C_{\frZ_1}) \arrow{r} & \frZ_1
    \end{tikzcd}
  \end{figure}
  
\noindent  is Cartesian follows by the description of the conductor ideals appearing in\cite[\href{https://stacks.math.columbia.edu/tag/035Q}{Lemma 035Q}]{StacksProject} and co-Cartesian since all the objects are reduced, and embed into the corresponding diagram for $\frX$. Hence it is enough to see that they have the same closed points as the desired pushout, which they do by construction.
  
Condition (3) follows from firstly the fact that $\frD_0$ is a complete intersection inside a Gorenstein scheme, hence is Gorenstein, using theorem 1 of~\cite{GorensteinFibrations}. The second statement follows from adjunction applied to the sequence of intersections $V(y_1), V (y_1, y_2), \ldots , \frD_0$. 
  
Condition (4) is precisely encoded by our assumptions on the structure of the log morphism $\frX \to \cS$.
  
\end{proof}

The final restriction that we require is that the family be simple. In the context of the Gross-Siebert program this is a monodromy requirement on the affine manifold. It constrains the monodromy polytopes associated to each strata to be elementary simplices, so the only integral points are the vertices.

\begin{dfn}

  A toric $k$-Tyurin degeneration is called \emph{simple} if the induced degeneration of $\frX$ is simple.
  
\end{dfn}

We could not answer satisfactorily if the induced degeneration of $\frD$ is necessarily also simple. If it is then we know that our construction here agrees with that appearing in~\cite{LogDegenerationI}. If it is not then we will see that it still carries a simple and generically log smooth log structure by restriction, and so we did not pursue this. All the examples we give of toric complete intersections indeed have this property.

Of course one is interested in providing examples which satisfy these conditions. The easiest examples of Tyurin degenerations of Calabi-Yau varieties are obtained as a $(k+1, 2)$ hypersurface in $\PP^k \times \PP^1$ and applying the subdivision techniques of~\cite{Shengda}. This corresponds to the polytope $(k+1)\Delta ^k \times 2 \Delta ^1$, the product of dilations of standard simplices.

\begin{eg}

  We partially follow~\cite{GSAndBB} and look for a product of MPCP desingularizations of the total family. Fix an MPCP subdivision of $(k+1)\Delta^k \subset \RR^k$ with associated strictly convex function $h_{\PP^k}$ and we implicitly pull this back to $\RR^k \oplus \RR$. On the other hand we have the piecewise linear function on $\RR$ given by $h_{\PP^1} = \max (0, x)$ where $x$ is a coordinate on the final factor, giving an MPCP desingularization of the polytope for $\PP^1$. Again we use the same notation for the pullback. The sum $h = h_{\PP^k} + h_{\PP^1} + \phi_{\PP^k \times \PP^1}$ gives a polytopal decomposition of $(k+1)\Delta ^k \times 2 \Delta ^1$ together with a strictly convex piecewise linear function on the refinement, such that $h - \phi_{\PP^k \times \PP^1}$ is nef. Rather than taking a one parameter graph to obtain a log smooth map to $\AA^1$ we take the two dimensional graph over $(h_{\PP^k}, h_{\PP^1})$, following section 4 of~\cite{Shengda} and for the same reasons outlined in example 4.2 of~\cite{LogDegenerationI} or theorem 3.10 of~\cite{GSAndBB} to obtain a generically log smooth morphism to $\AA^2$ throughout. The restriction of the two parameter degeneration to the diagonal is by construction the single parameter degeneration associated to $h$. On the dual side we choose the data of any MPCP desingularization none of whose faces have interior intersecting the hyperplane $H$ given by the vanishing of the final coordinate. This is not quite the data of an MPCP desingularization on $(k+1)\Delta ^k \times 2 \Delta ^1$, there are some ``square'' polytopes appearing.

  Nonetheless we follow the proof of simplicity given in Theorem 3.17 of~\cite{GSAndBB}, translated back to be working on $\Delta$ rather than $\nabla$. First we note that the local monodromy description the simplifies the barycentric subdivision description of the singular locus, it is supported on the set
  \[\cS ing ((k+1) \Delta^{k+1}) \times [0,1] \cup (k+1) \Delta^{k+1} \times \{ -1/2, 1/2\} \]
  In this decomposition there are four types of $\tau \in \Delta'$. The types are:

  \begin{enumerate}
  \item Faces intersecting the interior of the the top and bottom copies of $(k+1)\Delta^k$.
  \item Faces contained inside the boundary of the top and bottom copies of $(k+1)\Delta^k$.
  \item Faces of the side which are horizontal contained inside the vanishing of the final coordinate.
  \item Faces of the side which are vertical projecting to the boundary of $(k+1)\Delta^k$.
  \end{enumerate}

  The first type are locally given by an MPCP desingularization and in particular are guaranteed to be simple. The second and third types of faces contain no vertical edges, and any horizontal codimension one component can be ignored when constructing the choices of $\Omega_{i}$ since there is not any monodromy around them. In the third case this reduces the description of the monodromy to that induced by the MPCP desingularization given above, restricted to the boundary, which remains MPCP.

The second type has two types of divisors, the horizontal divisors of $(k+1) \Delta^k$ and the vertical fibres over codimension two strata in the decomposition of $(k+1) \Delta^k$. A local computation shows that the monodromy is only non-trivial for the vertical divisors contained inside the boundary, all the other loops occur entirely inside the interior of each face where there is trivial monodromy. Then we are back in the above situation.
  
The final type of faces are those which are vertical. These are not contained inside any horizontal divisor, and hence we may ignore horizontal edges in the calculation of the monodromy since they will be monodromy invariant. The new singularities that it meets have monodromy:
\[ n \mapsto n + \langle m, n \rangle \: u\]
where $u$ is a unit vector in the vertical direction and $m$ is the normal vector to the vanishing of $x$ in the $u$-direction. It is clear from this description that the associated monodromy polytopes are elementary simplices.

\end{eg}

In the case of a degeneration of a $K3$ surface embedded as an anti-canonical hypersurface inside $\PP^2 \times \PP^1$ we can draw explicitly the picture, seen in the figure below with the degenerating hyperplane dotted, and the 16 visible singularities of the 24 focus-focus singularities are marked in red. It is clear that the induced subdivision is simple, although we avoid specifically this surface case later:

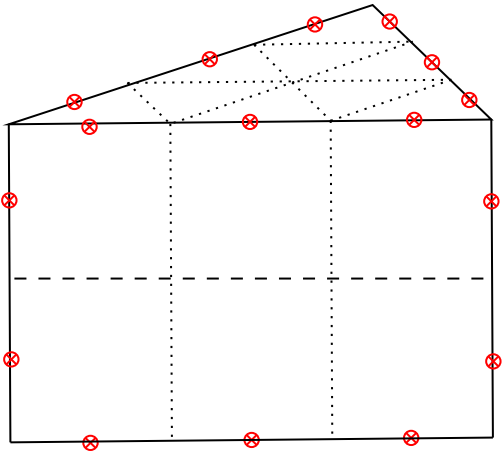
\begin{figure}[h]
\centering
\tikzset{every picture/.style={line width=0.75pt}} 
\begin{tikzpicture}[x=0.75pt,y=0.75pt,yscale=-2,xscale=2]

\draw   (360.47,90.7) -- (390.07,119.5) -- (269.7,120.7) -- cycle ;
\draw    (269.7,120.7) -- (270.1,200.7) ;
\draw    (390.07,119.5) -- (390.47,199.5) ;
\draw  [dash pattern={on 0.84pt off 2.51pt}]  (310,120.6) -- (310.4,200.6) ;
\draw  [dash pattern={on 0.84pt off 2.51pt}]  (350,119.8) -- (350.4,199.8) ;
\draw  [dash pattern={on 0.84pt off 2.51pt}]  (299.3,110.3) -- (310,120.6) ;
\draw  [dash pattern={on 0.84pt off 2.51pt}]  (330.9,100.7) -- (350,119.8) ;
\draw  [dash pattern={on 0.84pt off 2.51pt}]  (299.3,110.3) -- (380.1,109.5) ;
\draw  [dash pattern={on 0.84pt off 2.51pt}]  (330.9,100.7) -- (370.5,99.9) ;
\draw  [dash pattern={on 0.84pt off 2.51pt}]  (380.1,109.5) -- (350,119.8) ;
\draw  [dash pattern={on 0.84pt off 2.51pt}]  (370.5,99.9) -- (310,120.6) ;
\draw  [dash pattern={on 4.5pt off 4.5pt}]  (271.1,159.5) -- (390.27,159.5) ;
\draw  [color={rgb, 255:red, 255; green, 0; blue, 0 }  ,draw opacity=1 ] (268.04,139.83) .. controls (268.04,138.84) and (268.84,138.04) .. (269.83,138.04) .. controls (270.82,138.04) and (271.63,138.84) .. (271.63,139.83) .. controls (271.63,140.82) and (270.82,141.63) .. (269.83,141.63) .. controls (268.84,141.63) and (268.04,140.82) .. (268.04,139.83) -- cycle ; \draw  [color={rgb, 255:red, 255; green, 0; blue, 0 }  ,draw opacity=1 ] (268.57,138.57) -- (271.1,141.1) ; \draw  [color={rgb, 255:red, 255; green, 0; blue, 0 }  ,draw opacity=1 ] (271.1,138.57) -- (268.57,141.1) ;
\draw  [color={rgb, 255:red, 255; green, 0; blue, 0 }  ,draw opacity=1 ] (388.29,140.08) .. controls (388.29,139.09) and (389.09,138.29) .. (390.08,138.29) .. controls (391.07,138.29) and (391.88,139.09) .. (391.88,140.08) .. controls (391.88,141.07) and (391.07,141.88) .. (390.08,141.88) .. controls (389.09,141.88) and (388.29,141.07) .. (388.29,140.08) -- cycle ; \draw  [color={rgb, 255:red, 255; green, 0; blue, 0 }  ,draw opacity=1 ] (388.82,138.82) -- (391.35,141.35) ; \draw  [color={rgb, 255:red, 255; green, 0; blue, 0 }  ,draw opacity=1 ] (391.35,138.82) -- (388.82,141.35) ;
\draw  [color={rgb, 255:red, 255; green, 0; blue, 0 }  ,draw opacity=1 ] (369.04,119.58) .. controls (369.04,118.59) and (369.84,117.79) .. (370.83,117.79) .. controls (371.82,117.79) and (372.63,118.59) .. (372.63,119.58) .. controls (372.63,120.57) and (371.82,121.38) .. (370.83,121.38) .. controls (369.84,121.38) and (369.04,120.57) .. (369.04,119.58) -- cycle ; \draw  [color={rgb, 255:red, 255; green, 0; blue, 0 }  ,draw opacity=1 ] (369.57,118.32) -- (372.1,120.85) ; \draw  [color={rgb, 255:red, 255; green, 0; blue, 0 }  ,draw opacity=1 ] (372.1,118.32) -- (369.57,120.85) ;
\draw  [color={rgb, 255:red, 255; green, 0; blue, 0 }  ,draw opacity=1 ] (328.09,120.1) .. controls (328.09,119.11) and (328.89,118.31) .. (329.88,118.31) .. controls (330.87,118.31) and (331.67,119.11) .. (331.67,120.1) .. controls (331.67,121.09) and (330.87,121.89) .. (329.88,121.89) .. controls (328.89,121.89) and (328.09,121.09) .. (328.09,120.1) -- cycle ; \draw  [color={rgb, 255:red, 255; green, 0; blue, 0 }  ,draw opacity=1 ] (328.62,118.83) -- (331.15,121.37) ; \draw  [color={rgb, 255:red, 255; green, 0; blue, 0 }  ,draw opacity=1 ] (331.15,118.83) -- (328.62,121.37) ;
\draw  [color={rgb, 255:red, 255; green, 0; blue, 0 }  ,draw opacity=1 ] (388.79,180.33) .. controls (388.79,179.34) and (389.59,178.54) .. (390.58,178.54) .. controls (391.57,178.54) and (392.38,179.34) .. (392.38,180.33) .. controls (392.38,181.32) and (391.57,182.13) .. (390.58,182.13) .. controls (389.59,182.13) and (388.79,181.32) .. (388.79,180.33) -- cycle ; \draw  [color={rgb, 255:red, 255; green, 0; blue, 0 }  ,draw opacity=1 ] (389.32,179.07) -- (391.85,181.6) ; \draw  [color={rgb, 255:red, 255; green, 0; blue, 0 }  ,draw opacity=1 ] (391.85,179.07) -- (389.32,181.6) ;
\draw  [color={rgb, 255:red, 255; green, 0; blue, 0 }  ,draw opacity=1 ] (268.54,179.83) .. controls (268.54,178.84) and (269.34,178.04) .. (270.33,178.04) .. controls (271.32,178.04) and (272.13,178.84) .. (272.13,179.83) .. controls (272.13,180.82) and (271.32,181.63) .. (270.33,181.63) .. controls (269.34,181.63) and (268.54,180.82) .. (268.54,179.83) -- cycle ; \draw  [color={rgb, 255:red, 255; green, 0; blue, 0 }  ,draw opacity=1 ] (269.07,178.57) -- (271.6,181.1) ; \draw  [color={rgb, 255:red, 255; green, 0; blue, 0 }  ,draw opacity=1 ] (271.6,178.57) -- (269.07,181.1) ;
\draw  [color={rgb, 255:red, 255; green, 0; blue, 0 }  ,draw opacity=1 ] (288.29,200.83) .. controls (288.29,199.84) and (289.09,199.04) .. (290.08,199.04) .. controls (291.07,199.04) and (291.88,199.84) .. (291.88,200.83) .. controls (291.88,201.82) and (291.07,202.63) .. (290.08,202.63) .. controls (289.09,202.63) and (288.29,201.82) .. (288.29,200.83) -- cycle ; \draw  [color={rgb, 255:red, 255; green, 0; blue, 0 }  ,draw opacity=1 ] (288.82,199.57) -- (291.35,202.1) ; \draw  [color={rgb, 255:red, 255; green, 0; blue, 0 }  ,draw opacity=1 ] (291.35,199.57) -- (288.82,202.1) ;
\draw  [color={rgb, 255:red, 255; green, 0; blue, 0 }  ,draw opacity=1 ] (328.49,200.1) .. controls (328.49,199.11) and (329.29,198.31) .. (330.28,198.31) .. controls (331.27,198.31) and (332.07,199.11) .. (332.07,200.1) .. controls (332.07,201.09) and (331.27,201.89) .. (330.28,201.89) .. controls (329.29,201.89) and (328.49,201.09) .. (328.49,200.1) -- cycle ; \draw  [color={rgb, 255:red, 255; green, 0; blue, 0 }  ,draw opacity=1 ] (329.02,198.83) -- (331.55,201.37) ; \draw  [color={rgb, 255:red, 255; green, 0; blue, 0 }  ,draw opacity=1 ] (331.55,198.83) -- (329.02,201.37) ;
\draw  [color={rgb, 255:red, 255; green, 0; blue, 0 }  ,draw opacity=1 ] (368.29,199.58) .. controls (368.29,198.59) and (369.09,197.79) .. (370.08,197.79) .. controls (371.07,197.79) and (371.88,198.59) .. (371.88,199.58) .. controls (371.88,200.57) and (371.07,201.38) .. (370.08,201.38) .. controls (369.09,201.38) and (368.29,200.57) .. (368.29,199.58) -- cycle ; \draw  [color={rgb, 255:red, 255; green, 0; blue, 0 }  ,draw opacity=1 ] (368.82,198.32) -- (371.35,200.85) ; \draw  [color={rgb, 255:red, 255; green, 0; blue, 0 }  ,draw opacity=1 ] (371.35,198.32) -- (368.82,200.85) ;
\draw  [color={rgb, 255:red, 255; green, 0; blue, 0 }  ,draw opacity=1 ] (288.04,121.33) .. controls (288.04,120.34) and (288.84,119.54) .. (289.83,119.54) .. controls (290.82,119.54) and (291.63,120.34) .. (291.63,121.33) .. controls (291.63,122.32) and (290.82,123.13) .. (289.83,123.13) .. controls (288.84,123.13) and (288.04,122.32) .. (288.04,121.33) -- cycle ; \draw  [color={rgb, 255:red, 255; green, 0; blue, 0 }  ,draw opacity=1 ] (288.57,120.07) -- (291.1,122.6) ; \draw  [color={rgb, 255:red, 255; green, 0; blue, 0 }  ,draw opacity=1 ] (291.1,120.07) -- (288.57,122.6) ;
\draw  [color={rgb, 255:red, 255; green, 0; blue, 0 }  ,draw opacity=1 ] (382.79,114.58) .. controls (382.79,113.59) and (383.59,112.79) .. (384.58,112.79) .. controls (385.57,112.79) and (386.38,113.59) .. (386.38,114.58) .. controls (386.38,115.57) and (385.57,116.38) .. (384.58,116.38) .. controls (383.59,116.38) and (382.79,115.57) .. (382.79,114.58) -- cycle ; \draw  [color={rgb, 255:red, 255; green, 0; blue, 0 }  ,draw opacity=1 ] (383.32,113.32) -- (385.85,115.85) ; \draw  [color={rgb, 255:red, 255; green, 0; blue, 0 }  ,draw opacity=1 ] (385.85,113.32) -- (383.32,115.85) ;
\draw  [color={rgb, 255:red, 255; green, 0; blue, 0 }  ,draw opacity=1 ] (362.93,94.85) .. controls (362.93,93.86) and (363.74,93.06) .. (364.73,93.06) .. controls (365.71,93.06) and (366.52,93.86) .. (366.52,94.85) .. controls (366.52,95.84) and (365.71,96.64) .. (364.73,96.64) .. controls (363.74,96.64) and (362.93,95.84) .. (362.93,94.85) -- cycle ; \draw  [color={rgb, 255:red, 255; green, 0; blue, 0 }  ,draw opacity=1 ] (363.46,93.58) -- (365.99,96.12) ; \draw  [color={rgb, 255:red, 255; green, 0; blue, 0 }  ,draw opacity=1 ] (365.99,93.58) -- (363.46,96.12) ;
\draw  [color={rgb, 255:red, 255; green, 0; blue, 0 }  ,draw opacity=1 ] (373.48,105.1) .. controls (373.48,104.11) and (374.28,103.31) .. (375.27,103.31) .. controls (376.26,103.31) and (377.06,104.11) .. (377.06,105.1) .. controls (377.06,106.09) and (376.26,106.89) .. (375.27,106.89) .. controls (374.28,106.89) and (373.48,106.09) .. (373.48,105.1) -- cycle ; \draw  [color={rgb, 255:red, 255; green, 0; blue, 0 }  ,draw opacity=1 ] (374,103.83) -- (376.53,106.37) ; \draw  [color={rgb, 255:red, 255; green, 0; blue, 0 }  ,draw opacity=1 ] (376.53,103.83) -- (374,106.37) ;
\draw  [color={rgb, 255:red, 255; green, 0; blue, 0 }  ,draw opacity=1 ] (318.04,104.33) .. controls (318.04,103.34) and (318.84,102.54) .. (319.83,102.54) .. controls (320.82,102.54) and (321.63,103.34) .. (321.63,104.33) .. controls (321.63,105.32) and (320.82,106.13) .. (319.83,106.13) .. controls (318.84,106.13) and (318.04,105.32) .. (318.04,104.33) -- cycle ; \draw  [color={rgb, 255:red, 255; green, 0; blue, 0 }  ,draw opacity=1 ] (318.57,103.07) -- (321.1,105.6) ; \draw  [color={rgb, 255:red, 255; green, 0; blue, 0 }  ,draw opacity=1 ] (321.1,103.07) -- (318.57,105.6) ;
\draw  [color={rgb, 255:red, 255; green, 0; blue, 0 }  ,draw opacity=1 ] (284.29,115.08) .. controls (284.29,114.09) and (285.09,113.29) .. (286.08,113.29) .. controls (287.07,113.29) and (287.88,114.09) .. (287.88,115.08) .. controls (287.88,116.07) and (287.07,116.88) .. (286.08,116.88) .. controls (285.09,116.88) and (284.29,116.07) .. (284.29,115.08) -- cycle ; \draw  [color={rgb, 255:red, 255; green, 0; blue, 0 }  ,draw opacity=1 ] (284.82,113.82) -- (287.35,116.35) ; \draw  [color={rgb, 255:red, 255; green, 0; blue, 0 }  ,draw opacity=1 ] (287.35,113.82) -- (284.82,116.35) ;
\draw  [color={rgb, 255:red, 255; green, 0; blue, 0 }  ,draw opacity=1 ] (344.29,95.58) .. controls (344.29,94.59) and (345.09,93.79) .. (346.08,93.79) .. controls (347.07,93.79) and (347.88,94.59) .. (347.88,95.58) .. controls (347.88,96.57) and (347.07,97.38) .. (346.08,97.38) .. controls (345.09,97.38) and (344.29,96.57) .. (344.29,95.58) -- cycle ; \draw  [color={rgb, 255:red, 255; green, 0; blue, 0 }  ,draw opacity=1 ] (344.82,94.32) -- (347.35,96.85) ; \draw  [color={rgb, 255:red, 255; green, 0; blue, 0 }  ,draw opacity=1 ] (347.35,94.32) -- (344.82,96.85) ;
\draw    (270.1,200.7) -- (390.47,199.5) ;

\end{tikzpicture}

  \end{figure}

This is the limit of the cases where explicit examples are tractable, and in dimension higher than two computer algebra packages would be needed to attempt any calculations. For a general refinement of a nef partition there are choices of lifts involved, we handle the symmetric blow-up case from~\cite{tciDHT} below, along with a non-symmetric example for the quintic threefold. For this we need to use an actual MPCP desingularization, but one well adapted to the choice of resolution.

\begin{thm}

  Let $P \subset M$ be a reflexive polytope with a nef partition $P = \sum_0^{n} \Delta_i$ and let $\Delta_0 = \Delta_a + \Delta_b$ be a refinement of this nef partition. Let $\tilde{P}$ be the polytope associated to blowing up $\Delta_a$ and $\Delta_b$, and $\tilde{\Delta_i}$ the associated partitions of this polytope.

  Then there is a simple toric-Tyurin degeneration associated to this data, with generic fibre a complete intersection in $P$ given by the nef partition $\tilde{\Delta}_a, \tilde{\Delta}_b, \tilde{\Delta}_1, \ldots, \tilde{\Delta}_n$.
  
\end{thm}

\begin{proof}
  
  The degeneration may be embedded into the total space associated to $P \times [-1,1]$ embedded inside $M \oplus \RR$ and we write $H$ for the hyperplane given by the vanishing of the final coordinate. The associated refined polytopes $\Delta_i$ are the trivial product $\Delta_i \times \{ 0 \}$ except for $\tilde{\Delta}_a, \tilde{\Delta}_b$ and $\tilde{\Delta}_0$. The polytope $\tilde{\Delta}_a$ is the product $\Delta_a \times [0,1]$ with $\tilde{P}$ while $\tilde{\Delta_b}$ is the product $\Delta_b \times [-1, 0]$. The intersection with $H$ gives a reflexive polytope of dimension one lower, in which the degeneration for $\frD$ is embedded.

  Choose an MPCP desingularization $(\Sigma', h)$ of which none of the interiors of the maximal polytopes intersect $H$. This can be obtained by taking any MPCP desingularization and taking the Minkowski sum of the dual polytope to $h$ and a large multiple of the degenerate polytope $\{(0,1), (0,-1)\}$. This produces a new strictly convex function on a new refinement of $\Sigma$, $\tilde{\Sigma}$. There are vertical edges under the projection to the final factor and a small perturbation horizontally of these ensures that $\tilde{\Sigma}$ is an elementary triangulation of the intersection with $H$ as proved in~\cite{DualPolyhedraAndMirrorSymmetryForCalabiYauHypersurfacesInToricVarieties}. A perturbation of the points not lying as end points of the vertical edges then ensures that the rest of the triangulation is elementary. We replace $h$ and $\Sigma'$ by these choices.

  Then we form the two parameter family given by taking the graph over $h'$ constructed in~\cite{GSAndBB} and $h_{\PP^1}$. This gives a log smooth degeneration of the total space. Restricting this to the diagonal we obtain the MPCP desingularization given by $h' + h_{\PP^1}$. The construction of the complete intersection is modified to produce a semi-stable degeneration given by $h'$ as done in section 4 of~\cite{Shengda}. The induced two parameter family then has the above toric degeneration properties by the same arguments that the one parameter family does. The induced degeneration of $\frD$ can already be seen to be simple, induced by the toric degeneration obtained by intersecting the above data with the hyperplane $H$. The MPCP desingularization restricts on each side of $H$ to give compatible MPCP desingularizations. 

  The piecewise linear functions $\phi_*$ are defined to be the maximum over the points in the corresponding dual polytope. By construction these are linear across $H$ and so there is no monodromy across these faces. Thus the Tyurin singular locus $\frD$ is transverse to the singular locus of the degenerating family.

\end{proof}

\begin{eg}

  Let $Q_t \subset \PP^4$ be a family of generic quintic threefolds degenerating to the union of a degree $i$ and a degree $5-i$ hypersurface. We construct the family produced by blowing up the degree $i$ hyperplane. 

  Recall from~\cite{Shengda} the construction of degenerations of toric varieties given by subdividing polytopes. The polytope $\Delta$ for $\PP^4$ is of course given by the convex hull of the columns of the following matrix contained in the vector space spanned by $u_1, \ldots, u_4$:

  \[  \begin{bmatrix}
-1 & & 4 & &-1 & &-1 & &-1 \\
-1 & &-1 & & 4 & &-1 & &-1 \\
-1 & &-1 & &-1 & & 4 & &-1 \\
-1 & &-1 & &-1 & &-1 & & 4 \\
  \end{bmatrix} \]

  We subdivide this polytope along the hyperplane $H$ given by $u_1 = i-1$ and take the piecewise linear function $\min (0, i - 1 - u_i)$. The associated degeneration of the quintic is into a degree $5-i$ hypersurface and the blow up of a degree $i$ hypersurface, recalling the description of the fan for the total space of a projective bundle from Example 7.3.5 of~\cite{CoxLittleSchenk} and the classification of section 4.2 of~\cite{Shengda}. It carries an embedding into a non-compact toric variety with polytope $\tilde{\Delta}$. The polytope $\tilde{\Delta}$ further carries a piecewise linear function $\phi$ pulled back from $\Delta$ which generically fibrewise gives a quintic hypersurface inside $\PP^4$. One can use an auxiliary choice of toric compactification and the proof of~\cite{DualPolyhedraAndMirrorSymmetryForCalabiYauHypersurfacesInToricVarieties} to construct an MPCP desingularization of $\Delta$ whose non-linear locus contains $H \cap \partial \Delta$. This condition ensures that it pulls back to give an MPCP desingularization of the bottom face of $\tilde{\Delta}$, though possibly with singularities along the side. As in the previous two examples gives a two parameter family log smooth over $\AA^2$, where one factor is given by the new coordinate introduced by Hu's construction (section 4.1.2 of~\cite{Shengda}) and the other factor is the toric degeneration parameter (the section after Observation 3.9 of~\cite{GSAndBB}). Since $\phi$ is linear across $H$ the Tyurin singular locus is transverse to the singularities of the total space. It is a toric-Tyurin degeneration for the same reasons that the one parameter family is a toric degeneration. The corresponding degeneration of $\frD$ is given by restricting the polytope to the preimage of $H$, in particular it is a toric degeneration.

\end{eg}

  Note that this is a slightly different construction from the one appearing in~\cite{tciDHT} which makes explicit the target in which the degeneration is embedded. The above example also produces torically degenerating components of the two components of the Tyurin degeneration, fitting into the examples considered in our third section.

Examples of toric $k$-Tyurin degenerations can be constructed as a $(k+1, 2, \ldots 2)$ hypersurface in $\PP^k \times \PP^1 \times \ldots \times \PP^1$, or as more complicated subdivisions arising in the above example. It seems to be an essentially hard problem to tell if a given Tyurin degeneration, connected in moduli to a large complex structure limit point, admits such a toric model. This is similar to the case for toric degenerations arising in the Gross-Siebert program more generally.

We want to deduce statements via tropical geometry, and so we need to recall the definition of the dual intersection complex and tropicalisation.

\begin{dfn}

  Let $\frX$ be a log scheme. Let $\cC$ be the category whose objects are in bijection with the log strata of $\frX$ and whose morphisms are given by inclusions of strata. There is a functor sending each object to the dual cone of the stalk of $\overline{\cM}_\frX$ over the corresponding strata and each morphism to the corresponding inclusion of faces. The cone complex $\Sigma_{\cC} (\frX)$ is defined to be the complex obtained by this gluing data.

  Given a map $\frX \to \Spec k^\dagger$ there is an induced map $\Sigma_\cC (\frX) \to \RR_{\geq 0}$ and the \emph{tropicalisation} or \emph{dual intersection complex}, $\Sigma(\frX)$ is defined to be the fibre over $1$. At the moment it only carries an affine structure on the interior of each cell. To provide an affine structure away from a set of codimension two we equip this with a fan structure around each zero-dimensional strata using the corresponding toric variety as done in the proof of Proposition 4.10 of~\cite{LogDegenerationI}.

  From this data and a choice of closed gluing data in the cone picture (see ~\cite{LogDegenerationI} Definition 2.3) there is a projective scheme obtained by gluing spectra of the affine tangent spaces, as outlined in~\cite{FromRealAffine}. We will write this $\Proj k [\Sigma (\frX)]$. 
  
\end{dfn}

We can identify an embedding of dual intersection complexes of $\Sigma(\frD)$ into the dual intersection complex $\Sigma (\frX_\Delta)$ and by doing so build an explicit surjection on the level of rings.

\begin{lem}
\label{Localembedding}
  There is a closed embedding of $\Sigma (\frD)$ into $\Sigma (\frX)$ inducing a surjection on integral tangent spaces and compatible with the fan structures at zero-dimensional points of $\Sigma(\frD)$. Further for any toric degenerating component $\frZ$ of the Tyurin degeneration there is an affine linear open embedding of an open subset of $\Sigma (\frZ, \partial \frZ)^\circ $ into $ \Sigma(\frX_{\Delta})^\circ$ compatible with the fan structure around each point of $\Sigma (\frZ, \partial \frZ)$. 
  
\end{lem}

\begin{proof}

  We calculate at a point $x \in \frD$, comparing the local structure to the local structure induced by treating these as points of $\frX$. The let $P$ be the cone of $\frX$ at $x$, it admits a collection of maps $\NN^k \oplus \NN \to P$. Taking the dual these give a collection of maps $\rho_\Sigma: P^\vee_{\RR} \to \RR_{\geq 0}^k \oplus \RR_{\geq 0}$. Let $+: \RR_{\geq 0}^k \oplus \RR_{\geq 0} \to \RR_{\geq 0}^2$ be the map sending $(p_1, \ldots p_k , q)$ to $(\sum p_i,q)$. The composite $+ \cdot \rho_\Sigma$ is the structure map to $\RR_{\geq 0}^2$. By construction the set $\Sigma_\frD$ is locally precisely the fibre of $\rho_\Sigma$ over $(1,1, \ldots ,1)$, whilst $\Sigma_\frX$ is locally the fibre of $+ \cdot \rho_\Sigma$ over $(1,1)$. This map is integral affine linear, and the ordering of $p_1, \ldots p_k$ is fixed since the monodromy action on the divisors $\frZ^i$ is trivial.

  This gives a description of $\Sigma_\frX$ as a fibration over a $k+1$ dilation of a $k$-simplex, $\Delta_x$, with $\Sigma_\frD$ the fibre over the central point. As a warning we point out that this does not need to be a trivial fibration, for instance it might rescale the integral affine structure.

  The statement for $\frZ$ is the following. At a point $x$ of $\frX$ lying on $\frZ$ there is a canonical inclusion $P_\frZ \to P_{\frX_\Delta}$ such that the image of the elements  $2p$ and $\sum y_i + p$ are equal (here it is important that we are using $P_{\frX_\Delta}$ and not $P_{\frX}$). This corresponds to a diagram on the duals:

    \begin{figure}[h]
    \centering
    \begin{tikzcd}
      P_\frZ^\vee \arrow{r} \arrow{rd} &  P_{\frX_\Delta}^\vee \arrow{d} \\
      & \RR_{\geq 0}
    \end{tikzcd}
  \end{figure}

    The local structure of $\Sigma (\frZ, \partial \frZ) $ and  $\Sigma (\frX_\Delta)$ is given as the fibre over $1$ of each of these spaces. Since they are the same dimension and the map is a continuous embedding it is locally an isomorphism. Now in a neighbourhood of each point lying on $\frZ$ this map is in fact an isomorphism by a comparison of the local fan structures. It extends then in a neighbourhood of $\Sigma (\partial \frZ)$, which by compactness we may take to include a positive integral neighbourhood of each point. This gives the desired statement.

    \end{proof}
    
    The reason that it is not a global isomorphism is because of the existence of phenomena of the type pictured below:

    \begin{figure}[h]
      \centering
\tikzset{every picture/.style={line width=0.75pt}} 

\begin{tikzpicture}[x=0.75pt,y=0.75pt,yscale=-1,xscale=1]

\draw    (219,139.5) -- (299.5,219.12) ;
\draw  [dash pattern={on 4.5pt off 4.5pt}]  (240.83,100.58) -- (299.5,219.12) ;
\draw    (380.33,141.08) -- (299.5,219.12) ;
\draw  [dash pattern={on 4.5pt off 4.5pt}]  (362.33,101.08) -- (299.5,219.12) ;
\draw  [dash pattern={on 0.84pt off 2.51pt}]  (211.75,40.58) -- (240.83,100.58) ;
\draw  [dash pattern={on 0.84pt off 2.51pt}]  (181.17,100.58) -- (219,139.5) ;
\draw  [dash pattern={on 0.84pt off 2.51pt}]  (419.33,101.58) -- (380.33,141.08) ;
\draw  [dash pattern={on 0.84pt off 2.51pt}]  (399.33,39.58) -- (362.33,101.08) ;
\draw    (299.5,219.12) -- (439.75,219.58) ;
\draw  [dash pattern={on 0.84pt off 2.51pt}]  (498.25,164.08) -- (439.75,219.58) ;
\draw  [dash pattern={on 0.84pt off 2.51pt}]  (439.75,219.58) -- (499.5,219.83) ;
\draw  [dash pattern={on 0.84pt off 2.51pt}]  (489.92,121.58) -- (439.75,219.58) ;
\draw [color={rgb, 255:red, 255; green, 0; blue, 0 }  ,draw opacity=1 ]   (219,139.5) -- (510.08,140.33) ;
\draw [color={rgb, 255:red, 255; green, 0; blue, 0 }  ,draw opacity=1 ]   (240.83,100.58) -- (531.92,101.42) ;
\draw [color={rgb, 255:red, 255; green, 0; blue, 0 }  ,draw opacity=1 ]   (240.83,100.58) -- (219,139.5) ;
\draw  [dash pattern={on 4.5pt off 4.5pt}]  (362.33,101.08) -- (380.33,141.08) ;

\draw (292,113) node [anchor=north west][inner sep=0.75pt]    {$\Sigma (\frX_\Delta)$};
\draw (427,113) node [anchor=north west][inner sep=0.75pt]    {$\Sigma (\frZ, \partial \frZ)$};

\end{tikzpicture}

      \end{figure}

    \noindent The picture shows the conical complexes, together with the level set corresponding to the tropicalisation. The tropicalisation of $\Sigma (\frX_\Delta)$ is only locally a closed subset, and it may have some other components missing that do not appear in $\Sigma (\frZ, \partial \frZ)$, but on the left hand side it is a local isomorphism.

Associated to this is a choice to $0^{th}$ order of an embedding of $\widecheck{\frD}$ into $\widecheck{\frX_\Delta}$.

\begin{lem}

There is an open subset of $\Proj k[\Sigma (\frX_\Delta)]$ fibred over $\GG_m^k$ with fibres inducing potentially different log structures on $\Proj k[\Sigma (\frD)]$. Over an open subset the fibres have simple log structure in the sense of~\cite{LogDegenerationII}.
  
\end{lem}

\begin{proof}

  Open gluing data for $\frX$ gives rise to open gluing data for $\frD$ by restriction. In turn this gives cone gluing data compatibly under restriction. Defining equations for $\Proj k[\Sigma (\frD_\Delta)]$ can locally be constructed by pulling back the integral tangent vectors from $\Delta_k$ using the fact that $\Sigma(\frD)$ is transverse to the walls of $\Sigma (\frX_\Delta)$. Locally this gives a choice of monomial along each face, and any two choices differ by a function invertible on the polytope cell connecting them. Given a choice of gluing data in the cone picture (as defined in~\cite{LogDegenerationI} Definition 2.3) on $\Proj k[\Sigma (\frX_\Delta)]$ we get induced gluing data on $\Proj k[\Sigma (\frD)]$ compatible with these choices. Note that a monoid relation $m = 0$ becomes after passing to the associated monoid ring $z^m = 1$, not $z^m = 0$. Fix one chart corresponding to a $k$-dimensional face of $\Sigma (\frX_\Delta)$ corresponding to a zero-dimensional strata of $\Sigma(\frD)$ and a relation $z^m = 1$, twisting by the gluing data gives a collection of relations $z^{m_\sigma} = a_\sigma$ in each chart compatible with one another under restriction. Since the monodromy action on $\Sigma(\frX)$ restricts to the monodromy action on $\Sigma (\frD)$ the associated sections of the normal bundle are non-vanishing on $\Sigma_\frD$ and linearly independent.

  On an initial chart rather than taking $z^m = 1$ as the defining relation one can choose $z^m = a$ to obtain a different embedding. Doing this separately for each of the $k+1$ generators gives a $\GG_m^{k+1}$ worth of choices, but a common rescaling identifies the associated ideals and so we get the desired claim. A generic choice is transverse to the singular locus and any other log strata and so the strict induced log structure is indeed simple and generically log smooth.

\end{proof}

A similar statement holds for the LG model:

\begin{lem}

Let $\frZ_i$ be a toric degenerating component. There is an open embedding of $\Proj k[\Sigma (\frZ_i, \partial \frZ_i)]$ into $\Proj k[\Sigma (\frX)]$.

\end{lem}

\begin{proof}

  This is a trivial consequence of the above open embedding on tropical spaces.
  
  \end{proof}

\section{Logarithmic deformation of pairs}

We now bootstrap the smoothing results of~\cite{CLM} and~\cite{FFR} from smoothing Calabi-Yau varieties to smoothing pairs of such varieties, and higher codimension. This uses an explicit description of the canonical class to calculate how one corrects the Maurer-Cartan equations and show compatibility. We should explain that this is not a trivial application of deformation theory since none of the deformations are locally trivial, and it applies in the case where $\frD$ is an elliptic curve. The classical statement is the following:

\begin{thm}

  Suppose that $(\widecheck{\frX},\widecheck{\frD})$ are a pair with $\widecheck{\frX}$ a Calabi-Yau variety of dimension at least 3, and $\widecheck{\frD}$ a smooth Cartier divisor on $\widecheck{\frX}$ which is itself Calabi-Yau. Equivalently the normal bundle of $\widecheck{\frD}$, $\cN_{\widecheck{\frD}/\widecheck{\frX}}$, is $\cO_{\widecheck{\frD}}$. Then $\widecheck{\frX}$ is fibred over $\PP^1$ with fibre class $\widecheck{\frD}$ and any deformation of $\widecheck{\frX}$ induces an unobstructed deformation of $(\widecheck{\frX},\widecheck{\frD})$. Hence any smooth Calabi-Yau obtained from deforming $\widecheck{\frX}$ is also fibred by Calabi-Yau divisors.
  
\end{thm}

A similar statement holds for $\widecheck{\frD}$ any smooth Calabi-Yau subvariety of $\widecheck{\frX}$ with trivial normal bundle. Under these conditions $\widecheck{\frX}$ is generically fibred by smooth Calabi-Yau varieties and the same remains true for deformations of $\widecheck{\frX}$.

In~\cite{CLM} the authors prove formal smoothability for any log scheme for which two Hodge-theoretic conditions hold. These are known to hold for maximal degenerations of Calabi-Yau manifolds, and~\cite{FFR} extended this to much larger classes of singular spaces, those which are toroidal families. In future work they intend to extend this to families not just over the standard log point, but more general bases but still with underlying scheme a point. Our technique here would extend in parallel to their work.

To begin with we have $\widecheck{\frX} \to \Spec k^\dagger$ a log toroidal family and a strictly embedded codim $k$ log subscheme $\widecheck{\frD} \subset \widecheck{\frX}$ such that the induced map $\widecheck{\frD} \to \Spec k^\dagger$ is log toroidal and the normal bundle of $\widecheck{\frD}$ inside $\widecheck{\frX}$ is trivial, isomorphic to $\cO_{\widecheck{\frD}}^k$. In particular local to $\widecheck{\frD}$ the log scheme $\widecheck{\frX}$ is isomorphic to $\widecheck{\frD} \times \AA^k$. This will give local models for $\widecheck{\frX}$ near $\widecheck{\frD}$ compatible with the local models of $\widecheck{\frD}$. We consider divisorial deformations of $\frX$ and $\frD$ with a compatible closed embedding.

\begin{lem}

  Let $V \to \Spec k^\dagger$ be an affine log toroidal family, $W \subset V$ a strictly embedded codimension $k$ closed subscheme transverse to all the log strata and such that the induced map to $\Spec k ^\dagger$ is also a log toroidal family. Then any two first order deformations of this data are isomorphic. 
  
\end{lem}

\begin{proof}

  One \'etale locally factors the problem with $V = \AA^k \times Y$ and $W =   \{0\} \times Y$ where $Y \to \Spec k^\dagger$ is a log toroidal family. Then this works precisely as in~\cite{LogDegenerationII}.
  
\end{proof}

\subsection{Compatible BV structures}

The key idea of the proof of our main theorem is that there are compatible BV structures between the deformation theories of $\Proj k[\Sigma (\widecheck{\frX}_0)]$, $\Proj k[\Sigma (\widecheck{\frD})]$, and $(\Proj k[\Sigma (\widecheck{\frX}_0)]$, $\Proj k[\Sigma (\widecheck{\frD})])$. We follow~\cite{Sernesi} in defining the log vector fields on $\widecheck{\frX}$ tangent to $\widecheck{\frD}$. In the case $k=1$ we could describe this using an additional log structure on $\widecheck{\frX}$, but since we want to handle all the cases at once we will stick to a classical description. For this subsection we will assume that $\widecheck{\frX}$ and $\widecheck{\frD}$ are globally log smooth, and will lift this restriction later.

\begin{dfn}

  Let $\widecheck{\frD}$ locally be defined by $V(t)$ inside $\widecheck{\frX}$. Strict \'etale locally we have that $\widecheck{\frX}$ is isomorphic to $\widecheck{\frD} \times \AA^1$ with $t$ the final coordinate. We take a local basis $\ddx{t}, \ddx{x_1} ... \ddx{x_{k-1}} \ldots \ddlog {x_{k}} \ldots  \ddlog {x_{n-1}}$ of $\Theta^1_{\widecheck{\frX}}$. The algebra of polyvector fields on $\widecheck{\frX}$ tangent to $\widecheck{\frD}$ is the subalgebra generated locally to $\widecheck{\frD}$ by $t \ddx{t}, \ddx{x_1} ... \ddx{x_{k-1}} \ldots \ddlog {x_{k}} \ldots  \ddlog {x_{n-1}}$ and the full algebra over other open sets. We write this $\Theta^i _{\widecheck{\frX} \langle \widecheck{\frD} \rangle}$ for the degree $i$ term, and $\bigwedge \Theta_{\widecheck{\frX} \langle \widecheck{\frD} \rangle}$ for the whole polyvector field algebra.

In general if $\widecheck{\frD}$ is an equisingular complete intersection inside $\widecheck{\frX}$ then we analogously take the algebra of polyvector fields on $\widecheck{\frX}$ tangent to $\widecheck{\frD}$ to be the subalgebra generated locally to $\widecheck{\frD}$ by $t_1 \ddx{t_1}, \ldots , t_1 \ddx{t_r}, \ldots t_r \ddx{t_r},  \ddlog{x_{r+1}} ... \ddlog{x_{k}} \ldots \ddx {x_{k}} \ldots  \ddx {x_{n}}$ and the full algebra over other open sets. We write this $\Theta^i _{\widecheck{\frX} \langle \widecheck{\frD} \rangle}$ for the degree $i$ term, and $\bigwedge \Theta_{\widecheck{\frX} \langle \widecheck{\frD} \rangle}$ for the whole polyvector field algebra.
  
\end{dfn}

This structure controls the deformations of the local models as described in~\cite{FFR}, and for the same reason as their Theorem 6.13, combined with our local rigidity.

From now on let us assume that the surjectivity and Hodge theoretic assumptions necessary to apply techniques of~\cite{CLM} hold for $\widecheck{\frX}$ and $\widecheck{\frD}$. In the case we are interested in both are maximal and so by~\cite{GrossSiebert} Theorem 4.1 we will see that these assumptions hold.

Note that if $\widecheck{\frD}$ is codimension greater than one this is not a collection of bundles, rather of sheaves. One worries at this point that much of the argument involved in proving smoothness works by studying cotangent rather than tangent vectors. Amazingly this is not a problem as we shall see. An easy local calculation shows that $\bigwedge \Theta_{\widecheck{\frX} \langle \widecheck{\frD} \rangle}$ is a subalgebra of $\bigwedge \Theta_{\widecheck{\frX}}$ and the inclusion is compatible with the BV operator:

\begin{lem}

The complex $\bigwedge \Theta _{\widecheck{\frX} \langle \widecheck{\frD} \rangle} $ with the Schouten-Nijenhuis bracket is a sub-Gerstenhaber algebra of $\bigwedge \Theta _{\widecheck{\frX}}$. The BV operator on $\bigwedge \Theta_{\widecheck{\frX}}$ restricts to a BV operator on $\bigwedge \Theta_{\widecheck{\frX} \langle \widecheck{\frD} \rangle}$.

\end{lem}

\begin{proof}

  The statements are all local in nature, so we may reduce to the case where we have local equations for $\widecheck{\frD}$. To show that the Schouten-Nijenhuis bracket restricts to a bracket on this sub-space we need only check it for degree one elements, so for vector fields. Let $t_1 \ldots t_r$ be local equations for $\widecheck{\frD}$, then we must check that the brackets $[t_i \ddx{t_j}, t_l \ddx {t_m}]$ lie inside the span of these:

  \[\left[ t_i \ddx{t_j}, t_l \ddx {t_m} \right] = t_i \delta _{j,l} \ddx{t_m} - t_l \delta_{m,i} \ddx{t_j}\]

  And so this forms a sub-Gerstenhaber algebra. Now let $\omega$ be a global top form $\bigwedge dt_i \wedge \bigwedge dx_i \wedge \bigwedge d \: log \; x_i$ inducing the BV operator. Then the composite $\llcorner \: \omega \cdot d \cdot \llcorner \: \omega$ maps:
  \begin{align*} \prod_{j \in J} t_j & f (t_1, \ldots,  x_{n-1}) \bigwedge_{j \in J} \frac{d}{dt_j}  \bigwedge_{i \in I} \frac{d}{dx_i}  \bigwedge_{k \in K} \ddlog{x_k} \mapsto
    \prod_{j \in J } t_j f \bigwedge _{j \in \hat{J}} dt_j  \bigwedge_{i \in \hat{I}} dx_i  \bigwedge_{k \in \hat{K}} d \: log \;x_k \\
    & \mapsto \sum_{l \in J} \left( f \: dt_l \prod_{j \in J, j \neq l} t_j \right) \bigwedge_{j \in \hat{J}} dt_j \bigwedge_{i \in \hat{I}} dx_i + \prod_{j \in J} t_j \sum_{l \in J} \left(\frac{df}{dt_l} dt_l\right) \bigwedge _{j \in \hat{J}} dt_j  \bigwedge_{i \in \hat{I}} dx_i  \bigwedge_{k \in \hat{K}} d \: log \;x_k \\
    & \quad + \prod_{j \in J } t_j \bigwedge _{j \in \hat{J}} dt_j \left( \sum_{i \in I} \frac{df}{dx_i} dx_i + \sum_{k \in K} \frac{df}{d\: log \; x_i} d\:log \; x_i \right) \bigwedge_{i \in \hat{I}} dx_i \bigwedge_{k \in \hat{K}} d \: log \;x_k \\
    & \mapsto \sum_{l \in J} \left( \left( f \prod_{j \in J, j \neq l} t_j + \frac{df}{dt_l} \prod_{j \in J} t_j \right) \bigwedge_{j \in J, j \neq l} \ddx{t_j} \right) \bigwedge_{i \in I} \frac{d}{dx_i}  \bigwedge_{k \in K} \ddlog{x_k} \\
  & \quad + \left(\sum_{l \in I} \frac{df}{dx_l}\bigwedge_{i \in I, i \neq l} \frac{d}{dx_i}  \bigwedge_{k \in K} \ddlog{x_k} + \sum_{l \in K} \frac{df}{d \: log \; x_l}\bigwedge_{i \in I} \frac{d}{dx_i}  \bigwedge_{k \in K, k \neq l} \ddlog{x_k} \right)\prod_{j \in J} t_j \bigwedge_{j \in J} \frac{d}{dt_j}\end{align*}
  where $\hat{A}$ is the complement of the indexing set $A$. So this BV operator preserves the subalgebra $\bigwedge \Theta_{\widecheck{\frX} \langle \widecheck{\frD} \rangle}$.
\end{proof}

The definition of $\Theta^i_{\widecheck{\frX} \langle \widecheck{\frD} \rangle}$ is precisely set up so that the restriction $\Theta^i_{\widecheck{\frX} \langle \widecheck{\frD} \rangle} \mid_{\widecheck{\frD}}$ admits a map to $\Theta_{\widecheck{\frD}}$, which sends any section deforming away from $\widecheck{\frD}$ to zero. We claim that there is a commutative diagram of BV algebras:

\begin{lem}

  The following diagram commutes as a diagram of BV algebras and the outer square is Cartesian as a diagram of complexes of sheaves on $\widecheck{\frX}$:

  \begin{figure}[h]
  \centering
  \begin{tikzcd}
    \Theta_{\widecheck{\frX} \langle \widecheck{\frD} \rangle} \arrow{r} \arrow{d} & \Theta_{\widecheck{\frX}} \arrow{d} \\ \Theta_{\widecheck{\frX} \langle \widecheck{\frD} \rangle}\mid_{\widecheck{\frD}} \arrow{r} \arrow{d} & \Theta_{\widecheck{\frX}}\mid_{\widecheck{\frD}} \arrow{d} \\ \Theta_{\widecheck{\frD}} \arrow{r} & \Theta_{\widecheck{\frX}}\mid_{\widecheck{\frD}}
  \end{tikzcd}
  \end{figure}
  
\end{lem}

\begin{proof}

  This is a local calculation and follows directly from the above coordinate description.
  
\end{proof}

We now make the assumption that the normal bundle of $\widecheck{\frD}$ inside $\widecheck{\frX}$ is globally trivial. This is a strong assumption satisfied by fibrations and without it we could only show that $\widecheck{\frD}$ is stable inside $\widecheck{\frX}$. Since we assume that $\widecheck{\frD}$ is Calabi-Yau and dimension at least two this assumption is also stable under infinitesimal deformations of $\widecheck{\frD}$ (since $\Ext^1 (\cO_{\widecheck{\frD}}^r, \cO_{\widecheck{\frD}}^r) = 0$). Note that by strictness of the embedding of $\widecheck{\frD}$ into $\widecheck{\frX}$ the log normal bundle coincides with the normal bundle.

\begin{dfn}

  Let $\widecheck{\frD} \subset \widecheck{\frX}$ be a regularly embedded subscheme. Let $f_1 \ldots f_r$ be local equations cutting out $\widecheck{\frD}$. We define the sheaf of graded algebras $\bigwedge \cN_{\widecheck{\frD}/\widecheck{\frX}}$ to be the exterior algebra generated in degree one by $\ddx{f_1} \ldots \ddx{f_r}$ supported on $\widecheck{\frD}$.
  
\end{dfn}

\begin{lem}

There is a split exact sequence on the level of sheaves:
\begin{equation} 0 \to \Theta^1_{\widecheck{\frD}} \to \Theta^1_{\widecheck{\frX}} \mid_{\widecheck{\frD}} \to \cN_{\widecheck{\frD}/\widecheck{\frX}} \to 0 \end{equation}
This induces a split exact sequence of sheaves of graded algebras:
\begin{equation}\label{splitsequence} 0 \to \bigwedge \Theta_{\widecheck{\frD}} \to \bigwedge \Theta_{\widecheck{\frX}} \mid_{\widecheck{\frD}} \to \bigwedge \cN_{\widecheck{\frD}/\widecheck{\frX}} \otimes \bigwedge \Theta_{\widecheck{\frD}} \to 0 \end{equation}
The left hand map of this second sequence is a map of Gerstenhaber algebras. The induced map $\bigwedge \Theta_{\widecheck{\frX}} \to \bigwedge \Theta_{\widecheck{\frD}}$ is a surjective map of BV algebras.
\end{lem}

\begin{proof}

  Again we can pass to a local chart where we have coordinates $t_i$ defining $\widecheck{\frD}$ inside $\widecheck{\frX}$ and $t_1 \ddx{t_1}, \ldots , t_1 \ddx{t_r}, \ldots t_r \ddx{t_r},  \ddlog{x_{r+1}} ... \ddlog{x_{k}} \ldots \ddx {x_{k}} \ldots  \ddx {x_{n}}$ local generators on $\widecheck{\frX}$. There is a clear injection of $\Theta_{\widecheck{\frD}}$ into $\Theta_{\widecheck{\frX}} \mid_{\widecheck{\frD}}$ with cokernel spanned by the $\ddx{t_i}$. This produces the exact sequence. It is split via the obvious map.
  
  \end{proof}

It is not immediate that the ``Hodge numbers'' $H^i (\Theta^j_{\widecheck{\frX}})$ are invariant under pullback along the Frobenius map $Fr: \widecheck{\frX}' \to \widecheck{\frX}$. To see why note that although the Frobenius map on $S$ is an isomorphism of schemes for a perfect field, it can never be an isomorphism of fine saturated log schemes for divisibility reasons. Fortunately the maps $\widecheck{\frX} \to S$ and $\widecheck{\frD} \to S$ are saturated and the Frobenius map integral, we will see in the next section that this is enough to ensure independence.

\subsection{Away from the smooth locus}

We have described these sheaves in the case that $\widecheck{\frX}$ and $\widecheck{\frD}$ are log smooth over $S$. We are about to apply this to the mirrors where there is no guarantee that the log structure is smooth or even globally defined. Instead we know that $\widecheck{\frX}$ and $\widecheck{\frD}$ are central fibres of log toroidal families. Let us introduce some notation:

Let $\widecheck{\frX} \to S$ be a log toroidal morphism with generically smooth locus $i_{\widecheck{\frX}}: U_{\widecheck{\frX}} \subset \widecheck{\frX}$ and complement $Z_{\widecheck{\frX}}$, $\widecheck{\frD}$ a Calabi-Yau cycle on $\widecheck{\frX}$ which is itself a log toroidal family over $S$ with respect to the induced log structure and $i_{\widecheck{\frD}}: U_{\widecheck{\frD}} \subset \widecheck{\frD}$ the generically smooth locus. Write $Fr_S: S \to S$ for the Frobenius map, $\widecheck{\frX}'$ for the fibre product $S \times_S \widecheck{\frX}$, $\widecheck{\frD}'$ for the product $S \times_S \widecheck{\frD}$. These are themselves log toroidal over $S$, with open smooth locus $U_{\widecheck{\frX}'}$ and $U_{\widecheck{\frD}'}$ respectively, which we may assume are the pullbacks of $U_{\widecheck{\frX}}$ and $U_{\widecheck{\frD}}$ respectively. They admit induced pullback maps $Fr_{\widecheck{\frX}'}: \widecheck{\frX}' \to \widecheck{\frX}$ and $Fr_{\widecheck{\frD}'}: \widecheck{\frD}' \to \widecheck{\frD}$. Since the maps are saturated the underlying scheme maps of $Fr_{\widecheck{\frX}'}$ and $Fr_{\widecheck{\frD}'}$ are isomorphisms of schemes (though not over $S$).

As in~\cite{FFR} we can take the pushforward of these sheaves from the log smooth locus to the whole space and show that the desired properties hold for them, even once pushed forward. In particular we know already that $\bigwedge \Theta_{\widecheck{\frX}/S}$ and $\bigwedge \Theta_{\widecheck{\frD}/S}$ are $Z_{\widecheck{\frX}}$-closed. This follows from the fact that the open inclusion of the smooth locus is affine and so in both cases pushforward and pullback are exact. In particular this gives for free that $\Theta^\bullet_{\widecheck{\frX} \langle \widecheck{\frD} \rangle / S}$ is $Z_{\widecheck{\frX}}$-closed since it is the kernel of a surjective map between two $Z_{\widecheck{\frX}}$-closed sheaves. All the constructions above for split exact sequences continue to hold by exactness. We finally prove that Frobenius independence holds.

\begin{lem}

  There are isomorphisms $H^i (\widecheck{\frX}, \Theta_{\widecheck{\frX}/S}) \cong H^i (\widecheck{\frX}', \Theta_{\widecheck{\frX}'/S})$, $H^i (\widecheck{\frD}, \Theta_{\widecheck{\frD}/S}) \cong H^i (\widecheck{\frD}', \Theta_{\widecheck{\frD}'/S})$ and $H^i (\widecheck{\frX}, \Theta_{\widecheck{\frX} \langle \widecheck{\frD} \rangle /S}) \cong H^i (\widecheck{\frX}', \Theta_{\widecheck{\frX}'\langle \widecheck{\frD}' \rangle /S})$ for all $i$.
  
\end{lem}

\begin{proof}

  By~\cite{Olsson} there are canonical isomorphisms $i_{U_{\widecheck{\frX}'}}^* \Omega^i _{U_{\widecheck{\frX}}/S} \cong \Omega^i _{U_{\widecheck{\frX}'} / S}$ and $i_{U_{\widecheck{\frD}'}}^* \Omega^i _{U_{\widecheck{\frD}}/S} \cong \Omega^i _{U_{\widecheck{\frD}'} / S}$. By\cite{FFR} these remain isomorphisms under pushforward to $\widecheck{\frX}'$ and $\widecheck{\frD}'$ respectively. Since $Fr_{\widecheck{\frX}'} $ and $Fr_{\widecheck{\frD}'}$ are isomorphisms on schemes the first two isomorphisms now hold. The final isomorphism holds then by applying the five lemma to the exact sequence appearing in~\ref{splitsequence}.
  
\end{proof}

To apply this we construct the pre-sheaves $^k PV ^n_{\widecheck{\frX}}$, $^k PV ^n_{\widecheck{\frD}}$ and $^k PV ^n_{\widecheck{\frX}\langle \widecheck{\frD} \rangle}$ as outlined in~\cite{FFR} section 13.1. There are canonical maps $^k PV ^n_{\widecheck{\frX} \langle \widecheck{\frD} \rangle} \to ^k PV ^n_{\widecheck{\frX}}$ and $^k PV ^n_{\widecheck{\frX} \langle \widecheck{\frD} \rangle} \to ^k PV ^n_{\widecheck{\frD}}$ which are morphisms of BV algebras. Given a $^k \phi$ solving the Maurer-Cartan equation on $^k PV ^n_{\widecheck{\frX} \langle \widecheck{\frD} \rangle}$ the image under each of these maps solves the corresponding Maurer-Cartan equations and we have compatible morphisms of Gerstenhaber algebras between the induced cohomology pre-sheaves.

\begin{cor}
  The induced morphisms $H^\bullet (^k PV_{\widecheck{\frX} \langle \widecheck{\frD} \rangle}) \to H^\bullet (^k PV_{\widecheck{\frD}})$ and $H^\bullet (^k PV_{\widecheck{\frX} \langle \widecheck{\frD} \rangle}) \to H^\bullet (^k PV_{\widecheck{\frX}})$ are morphisms of presheaves of Gerstenhaber algebras.
\end{cor}

So the final thing we need to check are the assumptions of~\cite{CLM} Theorem 5.5, or equivalently~\cite{FFR} Theorem 13.4 and show that one can use the BV operator to construct solutions to the Maurer-Cartan equation. The components are degeneration of the Hodge-to-de-Rham spectral sequence at $E_1$ and a surjection $H^i (^k PV^\bullet) \to H^i (^{k-1} PV^\bullet)$.

\begin{thm}

  The ranks of $H^i (\Theta^j_{\widecheck{\frX} \langle \widecheck{\frD} \rangle})$ are invariant in one dimensional log toroidal families.

\end{thm}

\begin{proof}

  We will make use of the technology developed in~\cite{FFR} together with some homological algebra. First we spread out to finite characteristic which we can do compatibly for $\widecheck{\frX}$ and $\widecheck{\frD}$ preserving the inclusion. We can therefore form the Frobenius twists $F_{\widecheck{\frX}}: \widecheck{\frX} \to \widecheck{\frX}'$ and $F_{\widecheck{\frD}}: \widecheck{\frD} \to \widecheck{\frD}'$ with $\widecheck{\frD}' \subset \widecheck{\frX}'$ a closed subscheme. Note that $\widecheck{\frX}, \widecheck{\frX}', \widecheck{\frD}$ and $\widecheck{\frD}'$ all have trivial dualizing sheaf and $F$ is a finite map. Then taking derived sheaf $\Hom$ and using the fact that $F_{\widecheck{\frX}}^!\cO_{\widecheck{\frX}'} = \omega^\circ_{\widecheck{\frX}} \otimes (F_{\widecheck{\frX}}^* \omega^\circ_{\widecheck{\frX}'})^\vee \otimes F_{\widecheck{\frX}}^*\cO_{\widecheck{\frX}'} [\dim \widecheck{\frX} - \dim \widecheck{\frX}'] = \cO_{\widecheck{\frX}}$ we have:

  \begin{align*} F_{\widecheck{\frX} \: *} \Theta^\bullet_{\widecheck{\frX}/S} & = \cH ^0 RF_{\widecheck{\frX} \: *} \Theta^\bullet_{\widecheck{\frX}/S}\\
    & = \cH ^0 RF_{\widecheck{\frX} \: *} R \Hom_{\widecheck{\frX}} (W^\bullet_{\widecheck{\frX}/S}, \cO_{\widecheck{\frX}}) \\
    & = \cH ^0 RF_{\widecheck{\frX} \: *} R \Hom_{\widecheck{\frX}} (W^\bullet_{\widecheck{\frX}/S}, F_{\widecheck{\frX}}^!\cO_{\widecheck{\frX}'})\\
    & = \cH ^0 R \Hom_{\widecheck{\frX}'} (RF_{\widecheck{\frX} \: *} W^\bullet_{\widecheck{\frX}/S}, \cO_{\widecheck{\frX}'})\\
    & = \Hom_{\widecheck{\frX}'} (F_{\widecheck{\frX} \: *} W^\bullet_{\widecheck{\frX}/S}, \cO_{\widecheck{\frX}'})\\
  \end{align*}
  and similarly for $\widecheck{\frD}$. Thus we have that pushforward and dualizing commute for $\Theta_{\widecheck{\frX}/S}$ and $\Theta_{\widecheck{\frD}/S}$ on $\widecheck{\frX}$ and $\widecheck{\frD}$ respectively. Unfortunately the dual of $\Theta^\bullet_{\widecheck{\frX} \langle \widecheck{\frD} \rangle / S}$ would have to be the derived dual, and the above argument would not apply. The Cartier isomorphisms for $\widecheck{\frX}$ and $\widecheck{\frD}$ induce dual decompositions in the derived category of $\widecheck{\frX}$:
  \[ F_{\widecheck{\frX} \: *} (\Theta_{\widecheck{\frX}/S}^\bullet) \cong \bigoplus \Theta^i_{\widecheck{\frX}'/S} [-i] \text{\quad and \quad} F_{\widecheck{\frD} \: *} (\Theta_{\widecheck{\frD}/S}^\bullet) \cong \bigoplus \Theta^i_{\widecheck{\frD}'/S} [-i] \]
  Then there is a partial morphism of distinguished triangles which we may complete to a morphism of distinguished triangles:
  \begin{figure}[h]
    \centering
    \begin{tikzcd}
      \: \arrow[r] & F_{\widecheck{\frX} \: *} \Theta^\bullet_{\widecheck{\frX} \langle \widecheck{\frD} \rangle/ S} \arrow[r] \arrow[d, dashed] & F_{\widecheck{\frX} \: *} \Theta^\bullet_{\widecheck{\frX} / S} \arrow[r] \arrow[d] & F_{\widecheck{\frD} \: *} \Theta^\bullet_{\widecheck{\frD}/ S} \arrow[r] \arrow[d] & \:  \\
      \: \arrow[r] &  \bigoplus \Theta^i_{\widecheck{\frX}' \langle \widecheck{\frD'} \rangle/ S}[-i] \arrow{r}{\iota} &  \bigoplus \Theta^i_{\widecheck{\frX}' / S} [-i] \arrow[r] & \bigoplus \Theta^i_{\widecheck{\frD}'/ S} [-i] \otimes \bigwedge \cN_{\widecheck{\frD}/\widecheck{\frX}} \arrow[r] & \:
    \end{tikzcd}
  \end{figure}

  \noindent The right hand square commutes by the explicit description of the Cartier isomorphism. The left hand map is a quasi-isomorphism since the other two vertical maps are. The map $\iota$ is the natural inclusion, hence inducing the trivial differential on the complex $\bigoplus \Theta^i_{\widecheck{\frX}' \langle \widecheck{\frD}' \rangle/ S}[-i]$. Thus we have a decomposition of $F_{\widecheck{\frX} \: *} \Theta^\bullet_{\widecheck{\frX} \langle \widecheck{\frD} \rangle/ S}$ as a complex with trivial differential without having to find the derived dual. Now we apply the same argument as found in section 4 of~\cite{Relevements} to this complex using the two spectral sequences, the ``Hodge spectral sequence'' and the ``conjugate spectral sequence'':
  \[ E_1^{i,j}: \quad R^i f_* \Theta^j _{\widecheck{\frX} \langle \widecheck{\frD} \rangle / S} \Rightarrow R^{i+j} f_* \Theta^\bullet _{\widecheck{\frX} \langle \widecheck{\frD} \rangle/S}\]
  and
  \[ _cE_2^{i,j}: \quad R^i f'_* \cH^j F_{\widecheck{\frX} \: *} \Theta^\bullet _{\widecheck{\frX} \langle \widecheck{\frD} \rangle / S} \Rightarrow R^{i+j} f_* \Theta^\bullet _{\widecheck{\frX} \langle \widecheck{\frD} \rangle/S}\]
  obtained as the Cartan-Eilenberg spectral sequence and the Grothendieck spectral sequence respectively. Note that this argument crucially relies on three facts, that $\Theta^i_{\widecheck{\frX}' \langle \widecheck{\frD} \rangle/ S}$ are flat over $S$, formation of cohomology is compatible with base change for these sheaves and that the ``Hodge numbers'' $R^i f_* \Theta^j _{\widecheck{\frX} \langle \widecheck{\frD} \rangle / S}$ are invariant under Frobenius. The first is true when $S$ is a one-dimensional thickening of a point since there are no sections which are torsion over $S$. The second follows from the same results in large enough characteristic for $\Theta^i_{\widecheck{\frX} / S}$ and $\Theta^i_{\widecheck{\frD} / S}$. The last one we have proved earlier.
  
  This gives us a proof whenever there is a Frobenius lifting, and the standard spreading out argument then proves the claim.
  
\end{proof}

  For the sake of completeness we translate the proof of section 4 of~\cite{Relevements} into English. We found it both surprising that this result has not been exposited elsewhere and was relatively unknown in the community, for instance it acts to replace the local Poincar\'e lemma assumption in~\cite{CLM} and reduces the work required in section 11 of~\cite{FFR}. Although we have rephrased the argument without mentioning the de Rham complex we have kept the thematic namings of ``Hodge'', ``conjugate'' and ``Cartier'' as they explain where the terms came from.

  \begin{thm}

    Let $f: X \to S$ be a flat finite type morphism of schemes of characteristic $p$, $S$ local Artinian with closed point $s$ and residue field $k$ and $\cF^\bullet_{X_T/T}$ a function associating to every object $T$ of $Sch/S$ a bounded complex on $X_T = X \times_S T$ of $S$-flat coherent sheaves and to every morphism $g: (f': T' \to S) \to (f:T \to S)$ an isomorphism
    \[g^* \cF^\bullet_{X_T/T} \to \cF^\bullet_{X_{T'}/T'} \]
    Suppose further that $F_* (\cF^\bullet_{X/S})$ is quasi-isomorphic to a complex with trivial differential, where $F$ is the relative Frobenius map $X \to X'$ and that there is a Cartier isomorphism $C^{-1}: \cF^i_{X'/S} \to H^i F_* (\cF^\bullet_{X/S})$ for every $i$.

    Then under these hypotheses the ``Hodge spectral sequence'' formed as the Grothendieck spectral sequence of the complex verifies $E_1^{ij} = E_\infty^{ij}$ for $i+j = n$ and the sheaves $E^{ij}_1 = R^j f_* \cF_{X/S}^i$ are locally free of formation compatible with all base change for $i+j = n$.
    
  \end{thm}

  \begin{proof}

    First note that the statement that $F_* (\cF^\bullet_{X/S})$ is quasi-isomorphic to a complex with trivial differential implies that the ``conjugate spectral sequence'' obtained as the Grothendieck spectral sequence for the composite of proper pushforward along the composable maps $X \to X' \to S$ degenerates at page $E_2$.
    
    Write $S = \Spec A$, since the sheaves $\cF^i_{X/S}$ are $S$-flat and $f$ is finite type, the complexes $Rf_* \cF^i_{X/S}$ are isomorphic to a bounded complex $K_i^\bullet$ of free $A$-modules. We write $f_T$ for the base change of $f$ to $T$. For every subscheme $T \subset S$ the complex $R f_{T \: *} \cF_{X_T/T}^i$ is the base change of $R f_{*} \cF_{X/S}^i$, and in particular restricting to $s \in S$ we have that $H^j (X_s, \cF_{X_s/s}^i)$ is equal to $H^j (K_i \otimes_A k)$. It follows by devissage that
    \begin{equation} \label{lengtheqn} \text{length} R f_{*} \cF_{X/S}^i \leq h^{ij} \: \text{length} \: A \end{equation}
    where $h^{ij} = \dim_k H^j (X_s, \cF_{X_s/s}^i)$ and we take length as an $A$-module, which are all finite as they are the proper pushforward of finite complexes of coherent sheaves. We may assume that $K_i^\bullet$ is minimal, so that the differentials on $K_i^\bullet \otimes_A k$ are all zero. Once we know that the inequality in~\ref{lengtheqn} is an equality then all the differentials in $K_i^\bullet$ are trivial and we have the desired local freeness and compatibility with base change.

    The convergence $E_1^{ij} = E_\infty^{ij}$ is equivalent to the equality:

    \begin{equation}
      \text{length} R^n f_* \cF^\bullet_{X/S} = \sum_{i+j = n} \text{length} R^j f_* \cF^i_{X/S}
        \end{equation}

    On the other hand $X_s'$, constructed from $X_s$ by an extension of the base field, has the same ``Hodge'' numbers $h^{ij}$ as $X_s$. Hence one obtains an equality 

      \begin{equation}
      \text{length} R^n f_* \cF^\bullet_{X/S} = \sum_{i+j = n} \text{length} R^j f'_* \cF^i_{X'/S}
      \end{equation}

      Already if $S$ is the spectrum of a field this equality implies the desired equality in~\ref{lengtheqn}. Now we prove the general case by induction on the nilpotency of the maximal ideal $\frm$ of $A$. Suppose that $N$ is positive integer and $\frm^N = 0$, and by the above we are free to assume that $N$ is at least two. Hence there exists $N'$ with $1 \leq N' < N$ and $p N' \geq N$. For such an $N'$ the Frobenius endomorphism of $S$, $F$, factorises over $T = \Spec (A_1)$, $A_1 = A / \frm^{N'}$:

    \begin{figure}[h]
    \centering
    \begin{tikzcd}
      X\arrow{r}{F} \arrow[swap]{rd}{f} & X' \arrow{r} \arrow{d}{f'} & X_T \arrow {r} \arrow{d}{f_T} & X \arrow{d}{f} \\
      & S \arrow{r} & T \arrow{r} & S
    \end{tikzcd}
  \end{figure}
      
      Applying the induction hypothesis to $X_T/T$ we see that $R^j f_*' \cF^i_{X'/S}$ are locally free of rank $h^{ij}$. The degeneration of the ``conjugate spectral sequence'' ensures in turn that $R^n f_* \cF^\bullet_{X/S}$ is locally free of rank $\sum_{i+j = n} h^{ij}$. Hence we must have equality in ~\ref{lengtheqn} and the desired result.
        
    \end{proof}
  
The surjectivity assumption follows as argued in Theorem 13.1 of~\cite{FFR} once one knows the compatibility of these sheaves with base change shown above.

\begin{cor}
\label{PairSmoothing}
  For $\widecheck{\frX}$ and $\widecheck{\frD}$ described above there is an analytic log toroidal family smoothing this pair.
  
\end{cor}

\begin{proof}

  By the above compatibility results, together with the description of how to build the deformed Maurer-Cartan equation this follows from~\cite{FFR} section 13.
  
\end{proof}

  This proof actually works without assuming that $\widecheck{\frD}$ has the cohomology type of a Calabi-Yau so if $H^i(\cO_{\widecheck{\frD}}) \neq 0$ for $i \neq 0,n-1$, it only requires the existence of a global top form. In our case since $\widecheck{\frD}$ arises as the central fibre of a toric degeneration of Calabi-Yau varieties, and so does have the cohomology type of a Calabi-Yau by~\cite{FFR}. The same proof applies however to subschemes which are \'etale covers of products of Abelian varieties and Calabi-Yau varieties. Since we have an analytic smoothing we can then apply the classical deformation argument to a smooth fibre to obtain:

\begin{cor}

  A smooth generic fibre of the smoothing of $\widecheck{\frX}$ is fibred by Calabi-Yau varieties in codimension $d$. 
  
\end{cor}

In particular this applies to the above examples of toric $k$-Tyurin degenerations from which we deduce the following corollary.

\begin{cor}
~\label{DegenerationsToFibrations}
  Let $\frX \to \cS$ be a simple toric $k$-Tyurin degeneration, then the mirror $\widecheck {\frX_\Delta}$ is fibred by Calabi-Yau subvarieties in codimension $k$. Further if the degeneration of $\frD$ is simple then the fibres are mirror to $\frD$.
  
\end{cor}

The dgLa described above suggests that there is a version of the scattering diagrams described in~\cite{ThetaFunctions} where the coefficients lie in this subalgebra. There is a notion of consistency based on these structures and the above section implies that there is an appropriate iterated method to ensure consistency. It would be a fundamental development to define a canonical scattering diagram for such a subalgebra.

This result also implies a (non-effective) variant of a quantum-Lefschetz theorem. Take the asymptotic limit of the Maurer-Cartan equation as outlined in~\cite{ScatteringFromMC} and relating the terms to counts of log invariants on $\frX$ and $\frD$ as outlined in~\cite{Intrinsic} we see that this expresses the log Gromov-Witten invariants for $\frD$ in terms of $\frX$. This is a broader context than that considered in~\cite{BNR}, there are no ampleness assumptions around. We would rather develop fully the theory of these scattering diagrams than make this explicit.

\section{Gluing Landau-Ginzburg models}

We now consider what happens for $k=1$ if we formed the Landau-Ginzburg mirror to $\frZ^0$ and $\frZ^1$ and then attempt to glue them, considering each as a divisor pair relative to $\frD$. We solve the problem satisfactorily and prove that this recreates the mirror to $\frX_\Delta$ under a necessary compatibility condition.

\subsection{Formation of Landau-Ginzburg models}

The deformation theoretic Gross-Siebert program described above is sadly not very well developed for the case of LG models. We believe that such a theory will appear in tandem with a comparison to the modern approach. To avoid reliance on forthcoming work we will give a partial description here of how to construct Landau-Ginzburg models using the old theory, which is convoluted to avoid issues of boundedness of the cells. As an upside we see a natural occurrence of a phenomenon mentioned in~\cite{KKP}.  

We have the tropicalisation of $(\frZ^0, \frD)$, an unbounded polytope with a well defined affine direction $u$ and a proper piecewise linear map to $[0, \infty)$. Assuming that $\frD$ is nef on $\frZ^0$ all the initial slab functions on this point non-negatively in the $u$ direction. We can form an affine manifold with boundary by restricting to the closed subset over $[0, 1]$. This admits an affine linear embedding into $\Sigma (\frX)$ by~\ref{Localembedding}.

  We choose the slab structure given by restriction, so there are no slabs contained inside the fibre over $1$. But by the above observations the scattering diagram is consistent to order zero. Therefore applying the construction of~\cite{FromRealAffine} to this now bounded affine manifold we obtain a pair $(\overline{\frW_0}, \partial \frW_0)$ and the open subscheme $\overline{\frW_0} \setminus \partial \frW_0$ admits a map to $\AA^1$ given by $u$. This fits into the philosophy of~\cite{KKP} where one expects to see a compactification at infinity with a canonical choice of differential form on the interior.

Such a compactification exists for higher rank Tyurin degenerations at least so long as each of the components themselves have nef anti-canonical bundle. We believe that the restriction to this case is a necessary first step to understanding the ``Mirror $P=W$ conjecture'' of~\cite{HKP} through the Gross-Siebert construction. The compactifications have the property that there are no broken lines passing in from infinity. In particular one obtains a compactification of the base $\cB$, $\overline{\cB}$ by an $snc$ divisor and a relative compactification of $\frW$ which is relatively log divisorial over $\overline{\cB}$ and possesses a strata of dimension $d$ over a $d$-dimensional strata of $\overline{\cB}$.
  
  \subsection{Relative deformation theory}
  
To prove the existence of the desired embedding we construct a relative deformation theory for the interior of the Landau-Ginzburg models. This is simple on the central fibre where we know precisely where the singular fibres are, the same techniques apply to relatively smooth families of Calabi-Yau varieties with minor modifications.

To begin with we have $\widecheck{\frX} \to \PP^1 \times \Spec \widehat {k[P]}$ a formally versal smoothing of a toric Calabi-Yau space fibred over $\PP^1$ and a fixed $\GG_m \subset \PP^1$ a dense affine open with fibres generically divisorial smoothings of $\frD_0$. We have an LG model $w_0: \frW_0 \to \AA^1 \times \Spec \widehat{k[P]}$ and again a dense open $\GG_m \subset \AA^1$ whose fibres are generically divisorial smoothings of $\frD_0$. Over $0 \in \Spec \widehat{k[P]}$ the superpotential $\frW_{0} \to \AA^1$ admits compatible open embeddings into $\widecheck{\frX}_0 \to \PP^1$. We write $.^\circ$ for the restriction of an object to the corresponding open $\GG_m$.

Our goal is to produce an open embedding of $\frW_0$ into $\widecheck{\frX}$. We begin by introducing relative toric CY spaces and our choice of formal deformation functor. 

\begin{dfn}

  Let $\cC_{\frB}$ be the category of log schemes strict over $\Spec k[\NN]$ whose underlying scheme is an infinitesimal extension of $\frB$ and $\frB$ a fixed affine scheme. This category inherits a notion of small extension from the notion for the underlying schemes, noting that all morphisms are strict.
  
  Let $\rho: \frY \to \frB^\dagger := \frB \times \Spec k^\dagger$ be a morphism of log schemes such that:

  \begin{enumerate}
    \item The log scheme $\frY$ is a toric Calabi-Yau space. 
    \item $\rho$ is flat as a morphism of schemes and $\rho$ is log smooth on the log smooth locus of $\frY$
    \item There are diagrams as appearing in ~\cite{LogDegenerationI} (2.1) but with $X^\dagger$ replaced by $X^\dagger \times \frB $.
  \end{enumerate}

  A relative divisorial deformation of $\rho$ over $\cA \in \cC_{\frB}$ is a map $\rho_{\cA} : \frY_{\cA} \to \cA$ and an isomorphism of the central fibre $\frY_{\cA} \times_{\cA} \Spec k^\dagger \cong \frY$ such that:

  \begin{enumerate}
    \item $\rho_\cA$ is flat as a morphism of schemes and log smooth on the restriction to the log smooth locus of $\frY$.
    \item There are diagrams as appearing in ~\cite{LogDegenerationI} (2.2) but with $X^\dagger_A := Y^\dagger \times_{\Spec k[\NN]} \Spec A^\dagger$ replaced by $X^\dagger_\cA := Y^\dagger \times_{\Spec k[\NN]} \cA$.
  \end{enumerate}

  The deformation functor $F$ we consider will be have $\frB \subset \GG_m$ the common dense open sending an element $\cA$ to the set of relative divisorial deformations over it, where two relative divisorial deformations are isomorphic if there is an isomorphism fixing the central fibre. The action on morphisms is given by pullback.
    
\end{dfn}

We first of all claim that this deformation problem has a deformation and obstruction sheaf on $\frB$. To be precise we mean that the analogue of Theorem 2.11 of~\cite{LogDegenerationI} holds, but with $\Theta_{\frY / k^\dagger}$ replaced everywhere with $\Theta_{\frY / \frB^\dagger}$. The proof is simply compatibility of these sheaves with base change and the previous result.

\begin {thm}

  Let $\cA'$ and $\cA$ be elements of $\cC_{\frB}$, with $\cA'$ a small extension of $\cA$ by an ideal $I$, finite over $\frB$. Let $a \in F (\cA)$ be a divisorial deformation of $\frY \to \frB^\dagger$. Then:

  \begin{enumerate}
  \item Let $\frY_{\cA'}$ be a lift of $\frY_\cA$ to $\cA'$. Then the set of log automorphisms of $\frY_{\cA'}$ fixing $\frY_{\cA}$ is:
    \[R^0 \rho_* (\frY, \Theta_{\frY / \frB} \otimes _\frB I)\]
  \item The set of equivalence classes of lifts $\frY_{\cA'} \to \cA'$ if non-empty is a torsor over:
    \[ R^1 \rho_* (\frY, \Theta_{\frY / \frB} \otimes _\frB I) \]
  \item An obstruction sheaf for the existence of liftings is:
    \[ R^2 \rho_* (\frY, \Theta_{\frY / \frB} \otimes _\frB I) \]
  \end{enumerate}
  
\end{thm}

Let us check that the relative log versions of the Schlessinger axioms hold.

\begin{thm}

  The following three axioms hold:

  \begin{enumerate}
  \item Let $\cA' \to \cA$ be a small extension and $\cB \to \cA$ be a surjective morphism, then the canonical map:
    \[ F (\cA' \times_\cA \cB) \to F (\cA') \times_{F (\cA)} F(\cB)\]
    is surjective.
  \item If $\cA = \frB$ and $\cA' = \frB + \cM$ for a finite module over $\frB$. Then the above map is bijective.
  \item For a finite $\frB$ module $M$ the set $F (\frB \oplus M)$ is a finite $\frB$-module.
  \end{enumerate}
  
\end{thm}

\begin{proof}

  The third item follows from the above theorem together with relative properness of $\rho$. The second follows from the first once one notices that isomorphisms are preserved under deformation, see Lemma 9.1 of~\cite{LogDeformationTheory}. Therefore it remains to prove the first. Suppose we are given the data of $\cA' \to \cA$ a small extension and $\cB \to \cA$ a surjective morphism. For objects $X'  \in F(\cA')$ and $X'' \in F(\cB)$, both restricting to $X \in F(\cA)$. Then we can fix closed immersions $X \to X'$ and $X \to X''$. We then define $\underline{X^*}$ to be the scheme whose structure sheaf is $\cO_{\underline{X'}} \times_{\cO_{\underline{X}}} \cO_{\underline{X''}}$. We give this the log structure given by $\cM_{X'} \times_{\cM_{X}} \cM_{M''}$, by the universal property of products it admits a map to $\cO_{\underline{X'}} \times_{\cO_{\underline{X}}} \cO_{\underline{X''}}$. That this is a log scheme with the desired restriction properties is proved in section 9 of~\cite{LogDeformationTheory}. Condition i) of the definition of relative divisorial deformations is trivial, flat deformations of log smooth morphisms remain log smooth. We briefly prove condition ii) of the definition, $X^*$ admits \'etale locally maps to $X^\dagger_{\cA' \times_\cA \cB}$ whose central fibres are \'etale. But flat deformations of an \'etale map remain \'etale and so we have the result.
  
\end{proof}

We cannot apply the machinery built by Artin directly, even if it were translated into logarithmic geometry for the simple reason that we do not know openness of versality for log deformations. Raffaele Caputo has some results in this direction using analytic techniques in the absolute case, see~\cite{Raffaele}. Instead we argue directly, that we can construct a versal family over an open affine subset of $\GG_m$ by inducing one from the deformations of $\widecheck{\frX}$.  

\begin{lem}

  Suppose that $H^1 (\Theta_{\widecheck{\frX}_0 / \Spec k}) \to H^1 (\Theta_{\widecheck{\frD}_0 / \Spec k})$ is surjective under a splitting of the morphism $\Theta_{\widecheck{\frX}_0 / \Spec k} \to \Theta_{\widecheck{\frX}_0 / \Spec k}\mid_{\frD_0} \to \Theta_{\frD_0 / \Spec k}$ for some initial choice of fibre. Then the space $\widecheck{\frX} \to \Spec k[\![P]\!]$ induces a versal family over some choice of $\frB \subset \GG_m$.
  
\end{lem}

We prove this in several steps, first we relate the absolute and relative deformation theories (depending of course on a choice of splitting).

\begin{lem}

  There is a non-canonical morphism $H^1 (\Theta_{\widecheck{\frX} / \Spec k}) \otimes_k k[x ^{\pm}] \to R^1 \pi_* (\Theta _{\frW_i / \cT_i}) ^\circ$.
  
\end{lem}

\begin{proof}

  First take the relative tangent triangle for the triple of morphisms $\widecheck{\frX} \to \PP^1 \to \Spec k^\dagger$, and push it forward under $\rho$. This produces an exact sequence on $\PP^1$:
  \[ R \rho_* \Theta_{\widecheck{\frX}/\PP^1} \to R \rho_* \Theta_{\widecheck{\frX}/\Spec k^\dagger} \to R \rho_* \rho^* \Theta_{\PP^1 / \Spec k^\dagger} \]
  applying the projection formula to the last term we see that the cokernel of $R^1 \rho_* \Theta_{\widecheck{\frX}/\PP^1} \to R^1 \rho_* \Theta_{\widecheck{\frX}/\Spec k^\dagger}$ is torsion supported on the singular fibres. In particular over $\GG_m$ it is split. Pulling back to $\GG_m$ and applying cohomology and base change we obtain a map:

  \[ (R^1 \rho_* \Theta_{\widecheck{\frX}/\PP^1}) ^\circ \cong R^1 \rho_{*}^\circ \Theta_{\widecheck{\frX}^\circ / \GG_m}\]

The Grothendieck spectral sequence gives a map $H^1 (\Theta_{\widecheck{\frX} / \Spec k^\dagger}) \to \Gamma (R ^1 \rho_* (\Theta_{\widecheck{\frX} / \PP^1}))$, and since $\GG_m$ is affine composing with a choice of splitting we get the desired map.

\end{proof}

By our assumptions on the surjectivity of the map $H^1 (\Theta_{\widecheck{\frX}_0 / \Spec k}) \to H^1 (\Theta_{\frD_0 / \Spec k})$ there is an open dense $\frB \subset \GG_m$ where the above map is surjective on fibres. This is our choice of $\frB$.

We lift this formal statement about differentials to show that we can glue the formal smoothings:

\begin{lem}

  Let $\frF \to \frB \times \Spec k[\![t]\!]$ be a divisorial family deforming a trivial $\frD_0$ bundle. Then there is a non-canonical map $\frB \times \Spec k[\![t]\!] \to \frB \times \AA (H^1 (\Theta_{\widecheck{\frX} / \Spec k^\dagger}))$.
 ~\label{GlobalToRelative} 
\end{lem}

\begin{proof}

  For $t=0$ there is a canonical such map. We now induct on the order, suppose that we have a choice of map $\frB \times \Spec k[\![t]\!]/\langle t^k \rangle \to \frB \times \AA (H^1 (\Theta_{\widecheck{\frX} / \Spec k}))$ for some $k$. Then from the above analysis the lift of $\frF$ to $\frB \times \Spec k[\![t]\!]/\langle t^k \rangle$ defines an element of $R^1 \rho_* (\frY, \Theta_{\frY / \frB} \otimes _\frB \frB)$, up to an element of $R^0 \rho_* (\frY, \Theta_{\frY / \frB} \otimes _\frB \frB)$. But the map $H^1 (\Theta_{\widecheck{\frX} / \Spec k}) \otimes_k k[x ^{\pm}] \to R^1 \rho_* (\frY, \Theta_{\frY / \frB} \otimes _\frB \frB)$ is surjective and so we may lift it to an element of $H^1 (\Theta_{\widecheck{\frX} / \Spec k}) \otimes_k k[x ^{\pm}]$.

  But this gives a lift of $\frF$ inside $\widecheck{\frX}$, hence an embedding $\frB \times \Spec k[\![t]\!]/\langle t^{k+1} \rangle \to \frB \times \AA (H^1 (\Theta_{\widecheck{\frX} / \Spec k}))$.
  
\end{proof}

We need only one more tool to prove the claimed theorem.

\begin{lem}

Suppose that $h^1 (\Theta_{\widecheck{\frD} / \Spec k}) = 1$. Then there is a formal automorphism of $\frB \times \Spec k[\![t]\!]$ such that the induced map $\frB \times \Spec k[\![t]\!] \to \frB \times \AA (H^1 (\Theta_{\widecheck{\frX} / \Spec k}))$ sends elements of $H^1 (\Theta_{\widecheck{\frX} / \Spec k})$ to pure powers of $t$.
  
\end{lem}

\begin{proof}

  This follows from the formal version of the inverse function theorem applied to generators of $H^1 (\Theta_{\widecheck{\frX} / \Spec k^\dagger})$.
  
\end{proof}

We can now prove the existence of the embedding:

\begin{thm}

  There is a formal open embedding of $\frW^0$ into $\widecheck{\frX}$ compatible with the fibration structure.
\label{ReconstructionByGluing}
\end{thm}

\begin{proof}

  Suppose we have a divisorial smoothing of $\frW^0$ over $\Spec k[\![t]\!]$. The same trick as~\ref{GlobalToRelative} produces a relative deformation of the fibres. Then we can find an embedding of the restriction of this smoothing to $\frB$ horizontally into the smoothings of $\widecheck{\frX}$ restricted to $\frB$. We take the corresponding degeneration of $\widecheck{\frX}$ over $\Spec k[\![t]\!]$. Then there is a common dense open $\frU$ of $\frW^i$ and $\widecheck{\frX}$ which we can identify.

  Remove from $\PP^1$ the set $\AA^1 \setminus \frB \cup \{ 0 \}$ to obtain an open subset of $\widecheck{\frX}$ and glue along the open dense set found above. This produces a family which is locally everywhere a divisorial deformation of $\widecheck{\frX}_0$ over $\Spec k[\![t]\!]$. But the family $\widecheck{\frX}$ is a versal deformation space for divisorial deformations of $\widecheck{\frX}_0$, and hence this lifts to an open embedding of $\frW^i$ into a family of deformations of $\widecheck{\frX}$, which is the desired gluing statement. 
  
\end{proof}

\begin{eg}

  The mirror to a quintic threefold, a sextic fourfold, etc, can be constructed as a specialisation of a family of Calabi-Yau's obtained by gluing the two Landau-Ginzburg models obtained as mirrors to the components of the Tyurin degeneration constructed by blowing up both sides evenly.

\end{eg}

The same technique works for higher rank LG models under some comparable restrictions so long as one works with formal complex analytic spaces. The restriction is introduced part way through the construction, after we have introduced some terminology.

\begin{con}

  Let $\frX \to \cS$ be a type $k+1$ toric-Tyurin degeneration. There are two associated tropical structures. The first is the tropicalisation of the type $k+1$ locus which defines an affine manifold with singularities $\Xi_{gen} = \Sigma (\frX_{gen})$ but without a fan structure around the vertices. The second is the tropicalisation of the central fibre, defining for us an affine manifold with singularities $\Xi_0 = \Sigma (\frX_{0})$ with a fan structure around the vertices. There is a canonical map $\rho_{\Xi}:\Xi_{gen} \to \Xi_0$ given by specialisation. We take the vanilla gluing data on $\Xi_0$ and $\Xi_{gen}$

  To this data we have a reduced reducible scheme given by taking the cone picture on $\Xi_{gen}$. We write $\cB_{gen} = \Proj k[\Xi_{gen}]$ for this scheme, noting that without any more integral affine structure we cannot deform this even to first order. The map $\rho_{\Xi}$ lifts to a map $\rho_{\frX}: \Proj k[\Xi_0] \to \Proj k[\Xi_{gen}]$. 
  
  We are now ready for our assumption. Choose a zero dimensional strata $b \in \Xi_{gen}$ corresponding to a component of the general fibre over the Tyurin degenerate locus $\frZ^0$. Suppose that the induced degeneration of $\frZ^0$ is toric, simple and the boundary divisors of $\frZ^0$ are nef. This furnishes us with a collection of top dimensional strata of $\Xi_{gen}$, $\sigma_1 , \ldots , \sigma_m$ corresponding to the zero-dimensional strata of $\frZ^0$ and intersecting at $b$. We write $\frD_i$ for the Calabi-Yau variety corresponding to $\sigma_i$, and we assume that all the induced degenerations are simple. Locally there is a morphism of affine manifolds with singularities $\rho_{\Sigma}: \Xi_{gen} \to \Sigma (\frZ^0, \partial \frZ^0)$.

  We can construct a higher rank $LG$ model as the mirror to $(\frZ^0, \partial \frZ^0)$, by taking the induced tropical map to be evaluation against the irreducible strata of $\partial \frZ^0$ and we write this $\frW \to \frB$. The base $\frB$ is obtained by gluing a collection of toric varieties along toric strata, and each $\sigma_i$ corresponds to a dense torus in a toric component. By assumption there is a unique zero dimensional strata which we are free to conflate with $b$. Over the interior of each $\sigma_i$ the map $\rho_{\frX}$ is a locally trivial $\frD_i$ bundle. In particular it is a relatively divisorial family. We assume that the map $H^1 (\Theta_{\widecheck{\frX}_0 / \Spec k}) \to \oplus H^1 (\Theta_{\frD_{i} / \Spec k})$ is surjective and $\dim \Spec k[P] \geq \sum h^1 (\Theta_{\frD_{i} / \Spec k})$.

  Restrict the higher rank LG model mirror to $\frW_0 \to \frB$ to an analytic open set $U_P$ containing $b$. Take an open set of $\cB$ overlapping $U_P$ only inside the $\sigma_i$, and so that the complement $\delta_P$ is codimension two. The fibre of $\rho_{\frX}$ over this set is then codimension also at least two inside the total space $\widecheck{\frX}_{0}$. The same argument above would show that we can glue these two complex analytic varieties except over a closed set of codimension two. This uniquely defines a global deformation of $\frX$ compatible with the above gluing. Hence we can extend it to a globally defined gluing since differentials extend uniquely across codimension two sets.

\end{con}

Our restriction on $H^1 (\Theta_{\widecheck{\frX}_0 / \Spec k^\dagger}) \to H^1 (\Theta_{\frD_0 / \Spec k ^\dagger})$ has an interpretation on the mirror side which is more easily checkable. Assume that $\frX$ is an ample complete intersection inside a Fano toric variety $T$ of dimension at least 3. Up to torsion we have an isomorphism $H^2 (\frX) \cong H^1 (\Omega_\frX)$. Then by the Lefschetz hyperplane theorem there is an isomorphism $H^2_\QQ (T) \cong H^2_\QQ (\frX)$. Via composition with restriction there is therefore a canonical morphism $H^2_\QQ(\frX) \to \oplus H^2_\QQ (\frD_i)$.

\begin{lem}

  In this situation the above condition holds so long as each degeneration of $\frD_i$ is simple and $H^2_\QQ(\frX) \to \oplus H^2_\QQ (\frD_i)$ is surjective.
  
\end{lem}

\begin{proof}

  We apply Theorem 3.22 of ~\cite{LogDegenerationII}.

  \end{proof}

Constructing examples of such degenerations seems to be a hard problem, and the initial assumptions on surjectivity are certainly not sharp, it rarely includes the case of even toric hypersurfaces! One family of examples can be constructed via subdividing the hypercube $[-1,1]^{k}$ along the coordinate hyperplanes and studying an anti-canonical section. This produces a $k$-Tyurin degeneration whose minimal strata are therefore points. There is a further toric degeneration given by scaling the toric parameters.

A better solution than classifying examples would be to extend the work of~\cite{CLM} and ~\cite{FFR} to the case of log toroidal families over bases where the log structure on the base is allowed to change. How much of the theory can be extended from the fibrewise theory is at present unclear. Nonetheless there are some forthcoming results of Doran, Kostiuk and You in which they prove gluing formulae for periods under such higher rank degenerations. Since their results do not require such strong assumptions one would hope that better technology would allow us to lift the restrictions here.

\bibliographystyle{abbrv}
\bibliography{Uploaded}

\noindent Lawrence J. Barrott \\
Boston College, Boston, USA \\
\href{mailto:barrott@bc.edu}{barrott@bc.edu}\\

\noindent Charles F. Doran \\
University of Alberta, Edmonton, Canada \\
Center for Mathematical Sciences and Applications, Cambridge, USA\\
\href{mailto:charles.doran@ualberta.ca}{charles.doran@ualberta.ca}

\end{document}